\documentclass[pdflatex,sn-mathphys-num]{sn-jnl}

\usepackage{graphicx}%
\usepackage{multirow}%
\usepackage{amsmath,amssymb,amsfonts}%
\usepackage{amsthm}%
\usepackage{mathrsfs}%
\usepackage[title]{appendix}%
\usepackage{xcolor}%
\usepackage{textcomp}%
\usepackage{manyfoot}%
\usepackage{booktabs}%
\usepackage{algorithm}%
\usepackage{algorithmicx}%
\usepackage{algpseudocode}%
\usepackage{listings}%

\usepackage[utf8]{inputenc}
\usepackage[labelsep=space]{caption}
\usepackage{enumitem}
\usepackage{comment}
\usepackage{rotating}

\usepackage{amsbsy, mathtools}
\usepackage{thmtools, thm-restate}

\DeclareMathOperator*{\argmin}{\arg\!\min}

\theoremstyle{thmstyleone}%
\newtheorem{theorem}{Theorem}[section]
\newtheorem{lemma}[theorem]{Lemma}

\begin{document}

\title{Approximation of the Proximal Operator of the $\ell_\infty$ Norm Using a Neural Network}

\author*[1]{\fnm{Kathryn} \sur{Linehan}}\email{kjl5t@virginia.edu}

\author[2]{\fnm{Radu} \sur{Balan}}\email{rvbalan@umd.edu}

\affil*[1]{\orgdiv{Research Computing}, \orgname{University of Virginia}, \orgaddress{\city{Charlottesville}, \postcode{22904}, \state{VA}, \country{USA}}}

\affil[2]{\orgdiv{Department of Mathematics}, \orgname{University of Maryland}, \orgaddress{\city{College Park}, \postcode{20742}, \state{MD}, \country{USA}}}

\abstract{Computing the proximal operator of the $\ell_\infty$ norm, $\textbf{prox}_{\alpha ||\cdot||_\infty}(\mathbf{x})$, generally requires a sort of the input data, or at least a partial sort similar to quicksort. In order to avoid using a sort, we present an $O(m)$ approximation of $\textbf{prox}_{\alpha ||\cdot||_\infty}(\mathbf{x})$ using a neural network.  A novel aspect of the network is that it is able to accept vectors of varying lengths due to a feature selection process that uses moments of the input data.  We present results on the accuracy of the approximation, feature importance, and computational efficiency of the approach.  We show that the network outperforms a ``vanilla neural network'' that does not use feature selection.  We also present an algorithm with corresponding theory to calculate $\textbf{prox}_{\alpha ||\cdot||_\infty}(\mathbf{x})$ exactly, relate it to the Moreau decomposition, and compare its computational efficiency to that of the approximation.}    


\maketitle

\section{Introduction}

Proximal operators frequently serve as building blocks for solving large optimization problems.  For example, proximal operators commonly arise in ADMM, the alternating direction method of multipliers (see e.g., \cite{parikh_boyd_2014, polson_etal_2015}), which can be used to solve problems such as lasso and group lasso \cite{boyd_etal_2011}, robust principal component analysis \cite{candes_etal_2011}, CUR matrix decomposition \cite{li_etal_2019}, and robust CUR decomposition \cite{cai_etal_2021}, which is an intersection between the prior two applications\footnote{The applications of ADMM are vast.  See \cite{boyd_etal_2011} and \cite{boyd_website} for more examples.}.   In fact, our interest in the proximal operator stems from studying convex optimization formulations for CUR matrix decomposition, e.g., \cite{bien_xu_mahoney_2010, li_etal_2019, mairal_etal_2011}.  CUR decomposes a matrix exactly, $\mathbf{X=CUR}$, or approximately, $\mathbf{X \approx CUR}$, such that $\mathbf{C}$ contains a subset of the columns of $\mathbf{X}$ and $\mathbf{R}$ contains a subset of the rows of $\mathbf{X}$; see \cite{hamm_huang1_2020} for conditions that guarantee an exact decomposition.  While $\mathbf{C}$ and $\mathbf{R}$ can contain any subsets of columns and rows, respectively, a CUR decomposition stemming from convex optimization selects columns and rows based on the solution(s) to an optimization problem(s).  

In this paper we focus on the proximal operator of the $\ell_\infty$ norm, $\textbf{prox}_{\alpha ||\cdot||_\infty}: \mathbb{R}^m \rightarrow \mathbb{R}^m$, defined as  
\begin{equation}
\label{eqn:proxop}
\textbf{prox}_{\alpha ||\cdot||_\infty}(\mathbf{x}) = \argmin_{\mathbf{y} \in \mathbb{R}^m} \left(\frac{1}{2}||\mathbf{y-x}||_2^2 + \alpha ||\mathbf{y}||_\infty \right)
\end{equation}
for $\mathbf{x} \in \mathbb{R}^m$ and $\alpha \geq 0$. While closed form solutions exist for some proximal operators, e.g., for the $\ell_1$ and $\ell_2$ norms, this is not true for the $\ell_\infty$ norm.  However, algorithms exist for its computation that use the Moreau decomposition and are $O(m \log m)$ time and expected $O(m)$ time (see e.g., \cite{condat_2016}, and Section \ref{sec:rel_work} for details).  In this work, we discuss theoretical properties of $\textbf{prox}_{\alpha ||\cdot||_\infty}(\mathbf{x})$, which lead to an $O(m \log m)$ algorithm for its exact computation.  We relate this approach to that which uses the Moreau decomposition to compute $\textbf{prox}_{\alpha ||\cdot||_\infty}(\mathbf{x})$.  Computing $\textbf{prox}_{\alpha ||\cdot||_\infty}(\mathbf{x})$ generally requires a sort of the input data, or at least a partial sort similar to quicksort. We present a neural network approach to approximate $\textbf{prox}_{\alpha ||\cdot||_\infty}(\mathbf{x})$ which does not require a sort and is computationally efficient with $O(m)$ complexity.  A novel piece of this work is that we use a preprocessing and feature selection process that allows the network to accept vectors of varying lengths.  This is a desired property since 1) $\textbf{prox}_{\alpha ||\cdot||_\infty}(\mathbf{x})$ can be computed for a vector of any length, and 2) this eliminates the need for multiple networks, i.e., one for each unique vector length in the input data. The data preprocessing for the network is inspired by the exact approach developed to compute $\textbf{prox}_{\alpha ||\cdot||_\infty}(\mathbf{x})$.  In addition, we compare the performance of the network with that of a ``vanilla neural network" that does not use a feature selection process.

The remainder of this paper is organized as follows. Related work is presented in Section \ref{sec:rel_work}. The theoretical properties, algorithm to exactly compute $\textbf{prox}_{\alpha ||\cdot||_\infty}(\mathbf{x})$, and  the relationship between this algorithm and the Moreau decomposition are presented in Section \ref{sec:alg}. The neural network approximation of $\textbf{prox}_{\alpha ||\cdot||_\infty}(\mathbf{x})$ and results on the approximation accuracy, feature importance, and computational efficiency including a comparison with the exact algorithm are presented in Section \ref{sec:nn}.  Section \ref{sec:nn} also includes a comparison of the network performance with that of a vanilla network.  The paper concludes with Section \ref{sec:conc}.  Throughout this work vectors are denoted by bold, lowercase letters, the vector $(|\mathbf{v}_1|, |\mathbf{v}_2|, \ldots, |\mathbf{v}_m| )$ is denoted as $|\mathbf{v}|$, and the set $\{1,2,\ldots,m\}$ is denoted as $[m]$.  The notation $\mathbf{x} = \mathbf{x} + \mu$ denotes element-wise addition, i.e., $\mathbf{x}_i = \mathbf{x}_i + \mu$ for $i \in [m]$.

\section{Related Work}
\label{sec:rel_work}

A thorough treatment of proximal operators including properties, examples, applications, and history is given in \cite{parikh_boyd_2014}.  In this section we will focus on work directly related to computing $\textbf{prox}_{\alpha ||\cdot||_\infty}(\mathbf{x})$ rather than proximal operators more broadly.  Typically, for $\alpha>0$, $\textbf{prox}_{\alpha ||\cdot||_\infty}(\mathbf{x})$  is calculated using the Moreau decomposition: 
\[
\textbf{prox}_{\alpha ||\cdot||_\infty}(\mathbf{x}) = \mathbf{x} - \alpha \mathcal{P}_{||\cdot||_1 \leq 1} (\mathbf{x}/\alpha), 
\]
where $\mathcal{P}_{||\cdot||_1 \leq 1}: \mathbb{R}^m \rightarrow \mathbb{R}^m$ is projection onto the $\ell_1$ ball, i.e., 
\[
\mathcal{P}_{||\cdot||_1 \leq 1}(\mathbf{v}) = \argmin_{||\mathbf{y}||_1 \leq 1} \frac{1}{2} ||\mathbf{y}-\mathbf{v}||^2_2.
\] 
Projection onto the $\ell_1$ ball can be easily calculated from projection onto the simplex \cite{duchi_etal_2008}.  Projection onto the simplex can be calculated (naively) using an $O(m\log m)$ algorithm originally published in \cite{held_etal_1974} and rediscovered in \cite{shalev-shwartz_singer_2006}.  The algorithm performs soft-thresholding on the input vector using an iteratively determined threshold.  Naively implemented, the iterative algorithm to determine the threshold requires a sort of the entries of the input vector; however, in \cite{duchi_etal_2008} an expected linear time, divide and conquer algorithm is presented which leverages the fact that a full sort of the input vector is not actually necessary to find the threshold.  In addition, other fast algorithms for projection onto the simplex exist; see e.g., \cite{condat_2016, perez_etal_2020}. A review of algorithms for projection onto the simplex is also given in \cite{condat_2016}. 

In the next section we present an approach for calculating $\textbf{prox}_{\alpha ||\cdot||_\infty}(\mathbf{x})$ that does not make use of the Moreau decomposition, but does use a strategy of iteratively computing a threshold similar to the projection onto the simplex algorithms presented in \cite{duchi_etal_2008, held_etal_1974, shalev-shwartz_singer_2006} and Algorithms 1 and 3 in \cite{condat_2016}.

\section{Computing $\textbf{prox}_{\alpha ||\cdot||_\infty}$}
\label{sec:alg}
We first present the intuition for the computation of $\textbf{prox}_{\alpha ||\cdot||_\infty}(\mathbf{x})$.  Then we present the corresponding theory and algorithm.  To compute $\textbf{prox}_{\alpha ||\cdot||_\infty}(\mathbf{x})$ from Equation \ref{eqn:proxop}, let $t = ||\mathbf{y}||_\infty$ for $\mathbf{y} \in \mathbb{R}^m$.  For a fixed $t$, we can uniquely determine the corresponding vector $\mathbf{y}$ that solves Equation \ref{eqn:proxop}, i.e., $\mathbf{y}_k = \sigma_t(\mathbf{x}_k)$, where $\sigma_{t}: \mathbb{R} \rightarrow \mathbb{R}$ is defined as
        \begin{equation}
        \label{eqn:prox_k}
         \sigma_{t}(\mathbf{x}_k) =
            \begin{cases}
                t, & \text{if}\ \mathbf{x}_k \geq t \\
                \mathbf{x}_k, & \text{if}\ |\mathbf{x}_k| < t \\
                -t, & \text{if}\ \mathbf{x}_k \leq -t.
            \end{cases}
        \end{equation} 

\noindent Hence, to calculate $\textbf{prox}_{\alpha ||\cdot||_\infty}(\mathbf{x})$, we solve 
\begin{equation}
\label{eqn:tau}
\tau = \argmin_{t \geq 0} \left( \frac{1}{2} \sum_{k=1}^{m} (\sigma_{t}(\mathbf{x}_k)-\mathbf{x}_k)^2 + \alpha t \right)
\end{equation}
and then let $[\textbf{prox}_{\alpha ||\cdot||_\infty}(\mathbf{x})]_k = \sigma_{\tau}(\mathbf{x}_k)$ for each $k \in [m]$.  Note that $\tau = || \textbf{prox}_{\alpha ||\cdot||_\infty}(\mathbf{x}) ||_\infty$.  There is not a closed form solution to Equation \ref{eqn:tau}; thus, in the remainder of this section, we analyze and provide an algorithm to solve Equation \ref{eqn:tau}.

\begin{restatable}{theorem}{psiThm}
    \label{thm:psi}
    For any $\alpha \in \mathbb{R}$ and $\mathbf{x} \in \mathbb{R}^m$, let $I_t = \{k \:| \: |\mathbf{x}_k| \geq t\}$, $\psi: \mathbb{R}^+ \cup \{0\} \rightarrow \mathbb{R}$ be defined as
\[
\psi(t) = \frac{1}{2} \sum_{k=1}^{m} (\sigma_t(\mathbf{x}_k)-\mathbf{x}_k)^2 + \alpha t = \frac{1}{2}\sum_{k \in I_t} (|\mathbf{x}_k| - t)^2 + \alpha t,
\]
and $\mathbf{s}$ be a permutation of $|\mathbf{x}|$ such that $\mathbf{s}_1 \geq \mathbf{s}_2 \geq \cdots \geq \mathbf{s}_m \geq 0$ and $\mathbf{s}_{m+1} = 0.$  Then, $\psi$ is $\mathcal{C}^1$ over $t > 0$ and strictly convex over $0 \leq t \leq ||\mathbf{x}||_\infty = \mathbf{s}_1$.  In addition, if $\tau = \argmin_{t \geq 0} \psi(t)$, then $\tau \leq ||\mathbf{x}||_\infty = \mathbf{s}_1$, and $\tau = 0$ if and only if $|| \mathbf{x} ||_1 \leq \alpha$.  
\end{restatable}

\begin{proof}
    See Appendix \ref{app:psi_pf}.
\end{proof}  

Theorem \ref{thm:psi} provides the basis for Algorithm \ref{alg:base_min} which computes $\textbf{prox}_{\alpha ||\cdot||_\infty}(\mathbf{x})$.  The main iteration of this algorithm finds the value $0 < t \leq \mathbf{s}_1$ such that $\psi'(t)=0$, i.e., $t = (\sum_{k \in I_{t}}|\mathbf{x}_k| - \alpha)/|I_{t}|$.  Due to the strict convexity of $\psi$, there is at most one value of $t$ for which this is true.  Clearly the algorithm is necessary since we do not know the elements of $I_{\tau}$ in advance. The complexity of Algorithm \ref{alg:base_min} is $O(m\log m)$ due to the required sort of the input vector. We can improve the complexity to $O(m)$ in expectation using a similar divide and conquer strategy to that used in \cite{condat_2016, duchi_etal_2008} for computing the projection onto the simplex.  The expected $O(m)$ algorithm is provided in Appendix \ref{app:dc_alg}. 

\begin{algorithm}
\caption{Computation of the proximal operator of the $\ell_\infty$ norm}
\begin{algorithmic}[1]
\Require $\mathbf{x} \in \mathbb{R}^{m}$, $\alpha \geq 0$
\Ensure $\textbf{prox}_{\alpha ||\cdot||_\infty}(\mathbf{x}) \in \mathbb{R}^{m}$ 
\If {$||\mathbf{x}||_1 \leq \alpha$}
    \State $\textbf{prox}_{\alpha ||\cdot||_\infty}(\mathbf{x}) = \mathbf{0}$
\Else
\State Let $\mathbf{s}$ be a permutation of $|\mathbf{x}|$ such that $\mathbf{s}_1 \geq \mathbf{s}_2 \geq \cdots \geq \mathbf{s}_m \geq 0$. Set $\mathbf{s}_{m+1} = 0$.    
\State $\gamma = 0$
\State $i = 1$
\While{$i \leq m$}
    \Comment Assume that $|I_t| = i$, i.e., $\mathbf{s}_{i+1} < t \leq \mathbf{s}_{i}$
    \State $\gamma = \gamma + \mathbf{s}_i$
    \State $j=1$
    \While{$i+j \leq m$ and $\mathbf{s}_{i+j} == \mathbf{s}_i$}
    \Comment Account for repeated elements in $\mathbf{s}$
         \State $\gamma = \gamma + \mathbf{s}_i$
         \State $j = j+1$
    \EndWhile
    \State $i = i+(j-1)$
    \State $t_0 = (\gamma - \alpha)/i$ 
    \Comment Find the minimum of $\psi(t)$, $(\sum_{k=1}^{i}\mathbf{s}_k - \alpha)/i$ 
    \If { $\mathbf{s}_{i+1} < t_0 \leq \mathbf{s}_{i}$} \Comment If $t_0$ satisfies the assumption from Line 7
        \State $\tau = t_0$
        \State \textbf{break}
    \EndIf
    \State $i=i+1$
\EndWhile
\State $\forall k \in [m]$ compute $[\textbf{prox}_{\alpha ||\cdot||_\infty}(\mathbf{x})]_k = \sigma_{\tau}(\mathbf{x}_k)$    
\EndIf
\State \Return $\textbf{prox}_{\alpha ||\cdot||_\infty}(\mathbf{x})$  
\end{algorithmic}
\label{alg:base_min}
\end{algorithm}

As mentioned in Section \ref{sec:rel_work}, a common method for calculating $\textbf{prox}_{\alpha ||\cdot||_\infty}(\mathbf{x})$ when $\alpha > 0$ is to use the Moreau decomposition: 
\[
\textbf{prox}_{\alpha ||\cdot||_\infty}(\mathbf{x}) = \mathbf{x} - \alpha \mathcal{P}_{||\cdot||_1 \leq 1} (\mathbf{x}/\alpha), 
\]
where $\mathcal{P}_{||\cdot||_1 \leq r}(\mathbf{v}): \mathbb{R}^m \rightarrow \mathbb{R}^m$ is the projection of $\mathbf{v}$ onto the $\ell_1$ ball of radius $r$, 
\[
\mathcal{P}_{||\cdot||_1 \leq r}(\mathbf{v}) = \argmin_{||\mathbf{y}||_1 \leq r} \frac{1}{2} ||\mathbf{y}-\mathbf{v}||^2_2.
\] 
By the Moreau decomposition and \cite[Proposition 2.1]{condat_2016}, which relates projection onto the $\ell_1$ ball to projection onto the simplex, we have
\begin{equation*}
    \begin{split}
        \textbf{prox}_{\alpha ||\cdot||_\infty}(\mathbf{x}) 
        & = \mathbf{x} - \alpha \mathcal{P}_{||\cdot||_1 \leq 1} (\mathbf{x}/\alpha) \\
        & = \mathbf{x} - \mathcal{P}_{||\cdot||_1 \leq \alpha} (\mathbf{x}), \\
        & = \mathbf{x} - 
            \begin{cases}
                \mathbf{x}, & \text{if}\ ||\mathbf{x}||_1 \leq \alpha \\
                 \text{sign}(\mathbf{x}) \circ \mathcal{P}_{||\cdot||_1 = \alpha} (|\mathbf{x}|),  & \text{else}\  \\
            \end{cases} \\
    \end{split}
\end{equation*}
where $\mathcal{P}_{||\cdot||_1 = \alpha} (\mathbf{v}): \mathbb{R}^m \rightarrow \mathbb{R}^m$ is the projection of $\mathbf{v}$ onto the simplex, i.e., 
\[
\mathcal{P}_{||\cdot||_1 = r}(\mathbf{v}) = \argmin_{\substack{||\mathbf{y}||_1 = r \\ \mathbf{y}_i \geq 0 }} \frac{1}{2} ||\mathbf{y}-\mathbf{v}||^2_2,
\]
and element-wise vector multiplication is denoted using $\circ$, e.g., $\mathbf{a} \circ \mathbf{b} = (\mathbf{a}_1\mathbf{b_1}, \ldots, \mathbf{a}_m\mathbf{b}_m)$.  Hence, computing $\textbf{prox}_{\alpha ||\cdot||_\infty}(\mathbf{x})$ using the Moreau decomposition is dependent on $\mathcal{P}_{||\cdot||_1 = \alpha} (|\mathbf{x}|)$.  The $O(m \log m)$ algorithm for projection onto the simplex, mentioned in Section \ref{sec:rel_work}, is provided for reference in Algorithm \ref{alg:simp}. 

\begin{algorithm}
\caption{Projection onto the simplex \cite{held_etal_1974, shalev-shwartz_singer_2006}}
\begin{algorithmic}[1]
\Require $\mathbf{x} \in \mathbb{R}^{m}$, $\alpha > 0$
\Ensure $\mathcal{P}_{||\cdot||_1 = \alpha} (\mathbf{x}) \in \mathbb{R}^{m}$ 
\State Sort $\mathbf{x}$ into $\mathbf{v}$ such that $\mathbf{v}_1 \geq \mathbf{v}_2 \geq \cdots \geq \mathbf{v}_m$.
\State $\omega = 0$
\For{$i = 1 \: \text{to} \: m$} 
    \Comment Find $i^* = \max_{1 \leq i \leq m} \{i \; | \; (\sum_{k=1}^i \mathbf{v}_k - \alpha)/i < \mathbf{v}_i \}$
    \State $\omega = \omega + \mathbf{v}_i$
    \State $t_0 = (\omega - \alpha)/i$ 
    \If { $t_0 < \mathbf{v}_{i}$} 
        \State $i^* = i$
        \State $\tau_{\text{simplex}} = t_0$
        \Comment $\tau_{\text{simplex}} = (\sum_{k=1}^{i^*} \mathbf{v}_k - \alpha)/i^*$
    \EndIf
\EndFor
\State $\forall k \in [m]$ compute $[\mathcal{P}_{||\cdot||_1 = \alpha} (\mathbf{x})]_k = \max \{\mathbf{x}_k - \tau_{\text{simplex}}, 0 \}$    
\State \Return $\mathcal{P}_{||\cdot||_1 = \alpha} (\mathbf{x})$  
\end{algorithmic}
\label{alg:simp}
\end{algorithm}

For $\mathbf{x} \in \mathbb{R}^m$ such that $|| \mathbf{x}||_1 > \alpha$, Algorithm \ref{alg:base_min} to compute $\textbf{prox}_{\alpha ||\cdot||_\infty}(\mathbf{x})$ and Algorithm \ref{alg:simp} to compute $\mathcal{P}_{||\cdot||_1 = \alpha} (|\mathbf{x}|)$ are similar in that a threshold is iteratively computed and then used to compute the appropriate quantity.  We show that these thresholds are actually equal.  

\begin{restatable}{theorem}{sametThm}
\label{thm:same_t}
    For $\alpha > 0$, let $\mathbf{x} \in \mathbb{R}^m$ be such that $||\mathbf{x}||_1 > \alpha$, $\tau$ be calculated as in Algorithm \ref{alg:base_min} for finding $\textbf{prox}_{\alpha ||\cdot||_\infty}(\mathbf{x})$, and $\tau_{\text{simplex}}$ as in Algorithm \ref{alg:simp} for finding $\mathcal{P}_{||\cdot||_1 = \alpha}(|\mathbf{x}|)$.  Then $\tau = \tau_{\text{simplex}}$.
\end{restatable}

\begin{proof}
    See Appendix \ref{app:samet_pf}.
\end{proof}

So, the same threshold is used to compute $\textbf{prox}_{\alpha ||\cdot||_\infty}(\mathbf{x})$ in Algorithm \ref{alg:base_min} and $\mathcal{P}_{||\cdot||_1 = \alpha} (|\mathbf{x}|)$ in Algorithm \ref{alg:simp}; however, how the threshold is applied to the input vector to compute the final result differs between these two algorithms. Let the common threshold be $\tau$.  By the Moreau decomposition and Theorem \ref{thm:same_t}, for $\alpha > 0$ we have
\begin{equation*}
    \begin{split}
        [\textbf{prox}_{\alpha ||\cdot||_\infty}(\mathbf{x})]_k 
        & = \mathbf{x}_k - 
            \begin{cases}
                \mathbf{x}_k, & \text{if}\ ||\mathbf{x}||_1 \leq \alpha \\
                 \text{sign}(\mathbf{x}_k) \max\{|\mathbf{x}_k| -\tau, 0\},  & \text{else}\  \\
            \end{cases} \\
        & = \begin{cases}
                0, & \text{if}\ ||\mathbf{x}||_1 \leq \alpha \\
                \sigma_{\tau}(\mathbf{x}_k),  & \text{else,}\  \\
            \end{cases} \\    
    \end{split}
\end{equation*}
which is our approach to computing $\textbf{prox}_{\alpha ||\cdot||_\infty}(\mathbf{x})$.

We can also relate the threshold computed in Algorithm \ref{alg:base_min} to the solution to the projection onto the $\ell_1$ ball.  The solution to $\mathcal{P}_{||\cdot||_1 \leq 1}(\mathbf{v})$ is 
\begin{equation*}
\label{eqn:proj_l1}
[\mathcal{P}_{||\cdot||_1 \leq 1}(\mathbf{v})]_k  = \text{sign}(\mathbf{v}_k) \max(|\mathbf{v}_k| - \lambda, 0) =  
    \begin{cases}
        \mathbf{v}_k - \lambda, & \text{if}\ \mathbf{v}_k \geq \lambda \\
        0, & \text{if}\ |\mathbf{v}_k| < \lambda \\
        \mathbf{v}_k + \lambda, & \text{if}\ \mathbf{v}_k \leq -\lambda, \\
    \end{cases}
\end{equation*}
where the threshold $\lambda$ is computed using the equation $\sum_{k=1}^m \max (|\mathbf{v}_k| - \lambda, 0) = 1$.
Thus, by the Moreau decomposition, for $\alpha > 0$ we have
\begin{equation*}
\label{eqn:prox_MD}
\begin{split}
[\textbf{prox}_{\alpha ||\cdot||_\infty}(\mathbf{x})]_k & = \mathbf{x}_k - \alpha 
    \begin{cases}
        \mathbf{x}_k/\alpha - \lambda, & \text{if}\ \mathbf{x}_k/\alpha \geq \lambda \\
        0, & \text{if}\ |\mathbf{x}_k/\alpha| < \lambda \\
        \mathbf{x}_k/\alpha + \lambda, & \text{if}\ \mathbf{x}_k/\alpha \leq -\lambda, \\
    \end{cases} \\
   & = 
    \begin{cases}
        \alpha \lambda , & \text{if}\ \mathbf{x}_k \geq \alpha \lambda \\
        \mathbf{x}_k, & \text{if}\ |\mathbf{x}_k| < \alpha \lambda \\
        -\alpha \lambda, & \text{if}\ \mathbf{x}_k \leq - \alpha \lambda \\
    \end{cases} \\
   & = \sigma_{\alpha \lambda}(\mathbf{x}_k), \\
    \end{split}
\end{equation*}
which is Equation \ref{eqn:prox_k} with $\tau = \alpha \lambda$.  Hence the threshold computed using Algorithm \ref{alg:base_min}, $\tau$, and that for computing $\mathcal{P}_{||\cdot||_1 \leq 1}(\mathbf{x})$, $\lambda$, are related as well.

\section{Computation by Neural Network}
\label{sec:nn}

The previous algorithms presented to compute $\textbf{prox}_{\alpha ||\cdot||_\infty}(\mathbf{x})$ (including Algorithm \ref{alg:eff_min} in Appendix \ref{app:dc_alg}) require a sort of the input vector or at least a partial sort of the input vector similar to quicksort. For $\mathbf{x} \in \mathbb{R}^m$ this leads to $O(m \log m)$ or expected $O(m)$ and worst case $O(m^2)$ computation time.  In this section we present an $O(m)$ method to approximate $\textbf{prox}_{\alpha ||\cdot||_\infty}(\mathbf{x})$ using a neural network. Specifically, we use the neural network to approximate $\tau$ without requiring a sort of the input vector.  A novel aspect of the network is that it is able to accept vectors of varying lengths due to a feature selection process that uses moments of the input data.  This property is desired since 1) $\textbf{prox}_{\alpha ||\cdot||_\infty}(\mathbf{x})$ can be computed for a vector of any length and 2) it allows the user to train one network instead of one per vector length in the dataset. We present results on the accuracy of the approximation, feature importance, and computational efficiency of the approach.  We also compare the network loss to that of a ``vanilla neural network'', which naively solves the problem and serves as a baseline comparison.

\subsection{Data Preprocessing and Feature Selection}

For this approach, data consists of $\mathbf{x}$-$\alpha$-$\tau$ triples, where $\mathbf{x}$ can be any length and $\tau = || \textbf{prox}_{\alpha ||\cdot||_\infty}(\mathbf{x}) ||_\infty$.  To start, we present a helpful theorem for data preprocessing. 
  
\begin{restatable}{theorem}{tauLinOp}
\label{thm:tau_linop}
    Let $\alpha > 0$, $\mu \in \mathbb{R}$, and $||\mathbf{x}||_1 > \alpha$.  If $\tau = || \textbf{prox}_{\alpha ||\cdot||_\infty}(\mathbf{x}) ||_\infty$, then $|| \textbf{prox}_{\, ||\cdot||_\infty}(|\mathbf{x}|/\alpha + \mu) ||_\infty = \tau/\alpha + \mu$.  
\end{restatable}

\begin{proof}
    See Appendix \ref{app:tau_linop_pf}.
\end{proof}

We note that $\tau$ is a function of two variables, $\tau(|\mathbf{x}|, \alpha)$. By Theorem \ref{thm:tau_linop}, we can effectively remove $\alpha$ as a parameter of the problem by scaling $\mathbf{x}$ by $1/\alpha$.  Thus $\tau$ can be viewed as a function of only $\mathbf{x}$, i.e., $\tau(|\mathbf{x}|)$, that is permutation invariant with respect to the order of elements in $\mathbf{x}$.  

\begin{restatable}{theorem}{tauCont}
\label{thm:tau_cont}
    $\tau(|\mathbf{x}|)$: $\mathbb{R}^m \rightarrow \mathbb{R}$ is continuous. 
\end{restatable}

\begin{proof}
    See Appendix \ref{app:tau_cont_pf}
\end{proof}

By Theorem \ref{thm:tau_cont} and the Stone-Weierstrass Theorem, we can approximate $\tau(|\mathbf{x}|)$ by a polynomial.  Since $\tau(|\mathbf{x}|)$ is permutation invariant, this polynomial is symmetric.  Hence, it can be written uniquely as a polynomial in the elementary symmetric polynomials by the Fundamental Theorem of Symmetric Polynomials, and thus uniquely as a polynomial in power sums by the Newton-Girard formulas \cite{balan_etal_2022}.  For $\mathbf{x} \in \mathbb{R}^m$, we refer to the $k$-th power sum, $\sum_{i=1}^{m}\mathbf{x}_i^k$, as the $k$-th moment of $\mathbf{x}$.  

This reasoning motivated our choices of preprocessing and feature selection for each $\mathbf{x}$-$\alpha$-$\tau$ triple as seen in Algorithm \ref{alg:features_pre}.   We note that for any vector $\mathbf{x}$ such that $\| \widehat{\mathbf{x}} \|_1 \leq 1$ (or equivalently, $\|\mathbf{x}\|_1 \leq \alpha$, where $\widehat{\mathbf{x}} = |\mathbf{x}|/\alpha$), empty values are returned.  Since $\tau=0$ for these vectors, we effectively remove them from the dataset since a prediction of $\tau$ is not needed. 

\begin{algorithm}
\caption{Data preprocessing and feature generation}
\begin{algorithmic}[1]
\Require $\mathbf{x} \in \mathbb{R}^{m}$, $\alpha$, $\tau$, $k$ \Comment $k$ is the number of moments to compute.
\Ensure $\mathbf{w} \in \mathbb{R}^{k+3}$, $\widehat{\tau} \in \mathbb{R}$
\State $\widehat{\mathbf{x}} = |\mathbf{x}|/\alpha$
\If {$\| \widehat{\mathbf{x}} \|_1 \leq 1$}
    \State ${\mathbf{w}} = \emptyset$, $\widehat{\tau} = \emptyset$  
    \Comment An empty vector and scalar are returned.
\Else
    \State $\mu = ||\widehat{\mathbf{x}}||_1/m$
    \State $\widehat{\mathbf{x}} = \widehat{\mathbf{x}} - \mu$ such that $\widehat{\mathbf{x}}_i = \widehat{\mathbf{x}}_i - \mu$ for $i \in [m]$.
    \Comment Center the data.
    \State $\mathbf{w}_1 = \min(\widehat{\mathbf{x}})$ and $\mathbf{w}_2 = \max(\widehat{\mathbf{x}})$
    \For {$j \in [k]$} \Comment{Calculate the moments.}
        \If {$j == 1$}
            \State $\mathbf{w}_3 = \frac{1}{m}\sum_{i=1}^m |\widehat{\mathbf{x}}_i|$
        \Else
            \State $\mathbf{w}_{j+2} = \sqrt[j]{\frac{1}{m}\sum_{i=1}^m \widehat{\mathbf{x}}_i^j}$
        \EndIf
    \EndFor
    \State $\mathbf{w}_{k+3} = \ln(m)$
    \State $\widehat{\tau} = \tau/\alpha - \mu$
\EndIf
\State \Return $\mathbf{w}$, $\widehat{\tau}$ 
\end{algorithmic}
\label{alg:features_pre}
\end{algorithm}

For a vector $\mathbf{x}$ of any length, Algorithm \ref{alg:features_pre} returns a set of features $\mathbf{w} \in \mathbb{R}^{k+3}$, where $k$, the number of moments to compute, is constant for all $\mathbf{x}$-$\alpha$-$\tau$ triples in the dataset.  In summary, for each $\mathbf{x}$-$\alpha$-$\tau$ triple, we scale and center $|\mathbf{x}|$ and $\tau$, creating $\widehat{\mathbf{x}}$ and $\widehat{\tau}$, and then compute the following features of $\widehat{\mathbf{x}}$: the minimum, maximum, $k$ moments, and the natural log of the length.  In order to keep all features on a similar scale, moments are scaled by $1/m$ and an appropriate root is taken, and the natural log of the length is used.  We note that for the first moment, we use $\frac{1}{m}\sum_{i=1}^m |\widehat{\mathbf{x}}_i|$ instead of $\frac{1}{m}\sum_{i=1}^m \widehat{\mathbf{x}}_i$ since $\frac{1}{m}\sum_{i=1}^m \widehat{\mathbf{x}}_i = 0$ due to the centering of the data.  The output of Algorithm \ref{alg:features_pre}, $\mathbf{w}$ and $\widehat{\tau}$, become input to the network during training.  The network will output an approximation to $\widehat{\tau}$.

\subsection{Numerical Experiments}

Experiments were run on the University of Virginia's High-Performance Computing system, Rivanna.  Rivanna's hardware includes multiple nodes and cores per node, and GPUs.  We used one (CPU) node, using 8 cores of an Intel(R) Xeon(R) Gold 6248 CPU at 2.50GHz and 72GB of RAM for all experiments unless otherwise noted.  

We performed six experiments using a neural network to predict $\widehat{\tau}$. The network architecture used in the experiments is given in Table \ref{tab:fnn_arch}.  This network is small enough that it did not require nor benefit from training on a GPU (the training time per epoch was longer when trained on the GPU rather than the CPU).  In each of the following experiments, the network was trained and tested using single precision PyTorch.  To train the network, we used the Adam optimizer with a learning rate of 0.001, and a batch size of 32.  The loss function used was mean squared error (MSE) as defined below, where $\widetilde{\tau}$ is the output of the network, i.e., the approximation to $\widehat{\tau}$, and $b$ is the batch size: 
\[
MSE = \frac{1}{b} \sum_{i=1}^{b} (\widetilde{\tau}_i - \widehat{\tau}_i)^2.
\]
In each experiment, Algorithm \ref{alg:features_pre} was used with $k=10$ moments to preprocess and generate features for each $\mathbf{x}$-$\alpha$-$\tau$ triple. The architecture and number of moments were chosen through experimentation with the goal of having high performance and low complexity. We found that increasing the number of hidden layers, and/or the number of neurons in each layer, and/or the number of moments did not lead to significant performance gains.

\begin{table}[h]
    \caption{Neural network architecture for experiments 1-6.  $k$ is the number of moments computed in Algorithm \ref{alg:features_pre}}
    \label{tab:fnn_arch}
    \begin{tabular}{l c c c c}
        \toprule
        \textbf{Layer} & \textbf{Input} & \textbf{1st Hidden} & \textbf{2nd Hidden} &  \textbf{Output} \\
        \midrule
        \textbf{Number of Neurons} & $k+3$ & 25 & 10 & 1 \\
        \textbf{Activation Function} & - & ReLU & ReLU & - \\
        \botrule
    \end{tabular}
\end{table}  

Data for each experiment consisted of 10,000 $\mathbf{x}$-$\alpha$-$\tau$ triples where the vector data was either Gaussian $\mathcal{N}(0,1)$ or uniformly distributed $\mathcal{U}(0,1)$ and vectors were of varying lengths (see Table \ref{tab:fnn_data} for details), $\alpha$'s were sampled from the $\mathcal{U}[1,6)$ distribution, and $\tau$'s were computed using Algorithm \ref{alg:base_min}.  Vector lengths were sampled from the discrete uniform distribution over [1,000, 2,000] or [1,000, 100,000].

\begin{table}[h]
    \caption{Summary of vector data for experiments 1-6}
    \label{tab:fnn_data}
    \begin{tabular}{cccl}
        \toprule
        \textbf{Ex} & \textbf{$\mathcal{N}(0,1)$ Vectors} & \textbf{$\mathcal{U}(0,1)$ Vectors} & \textbf{Vector Lengths}  \\
        \midrule
        1 & 10,000 & 0 & 1,000 - 2,000  \\
        2 & 10,000 & 0 & 1,000 - 100,000  \\
        3 & 0 & 10,000 & 1,000 - 2,000  \\
        4 & 0 & 10,000 & 1,000 - 100,000  \\
        5 & 5,000 & 5,000 & 1,000 - 2,000  \\
        6 & 5,000 & 5,000 & 1,000 - 100,000  \\
        \botrule
    \end{tabular}
\end{table}

Of the 10,000 $\mathbf{w}$-$\widehat{\tau}$ pairs generated by Algorithm \ref{alg:features_pre}, we used an 80-20 random train-test split for experiments 1-4 and an 80-20 random train-test split stratified on vector data distribution for experiments 5-6. 

\subsubsection{Learning Curves and Comparison with ``Vanilla Neural Network''}

The learning curves for experiments 1-6 are presented in Figure \ref{fig:fnn_lc}\footnote{The color palette used for all learning curve plots is cubehelix \cite{green_2011}.}. The training (line a) and testing (line b) errors for $\widehat{\tau}$ are the errors for the network output.  For a given epoch, the reported training error is the average error per batch. Testing was completed in a single batch, so the reported testing error is the MSE for the testing set.   For a given $\mathbf{x}$-$\alpha$-$\tau$ triple, an approximation to $\tau$ is $\alpha(\widetilde{\tau} + \mu)$, where $\widetilde{\tau}$ is the network output, and $\mu$ is computed during Algorithm \ref{alg:features_pre}.  Thus, we also present the testing error on $\tau$ (line c).  For experiments 5-6, Figure \ref{fig:fnn_lc} also includes the $\widehat{\tau}$ testing error for the Gaussian $\mathcal{N}(0,1)$ vector data (line b$_1$) and the uniformly distributed $\mathcal{U}(0,1)$ vector data (line b$_2$) separately.

\begin{sidewaysfigure}
    \centering
    \includegraphics[width=1.0\textwidth]{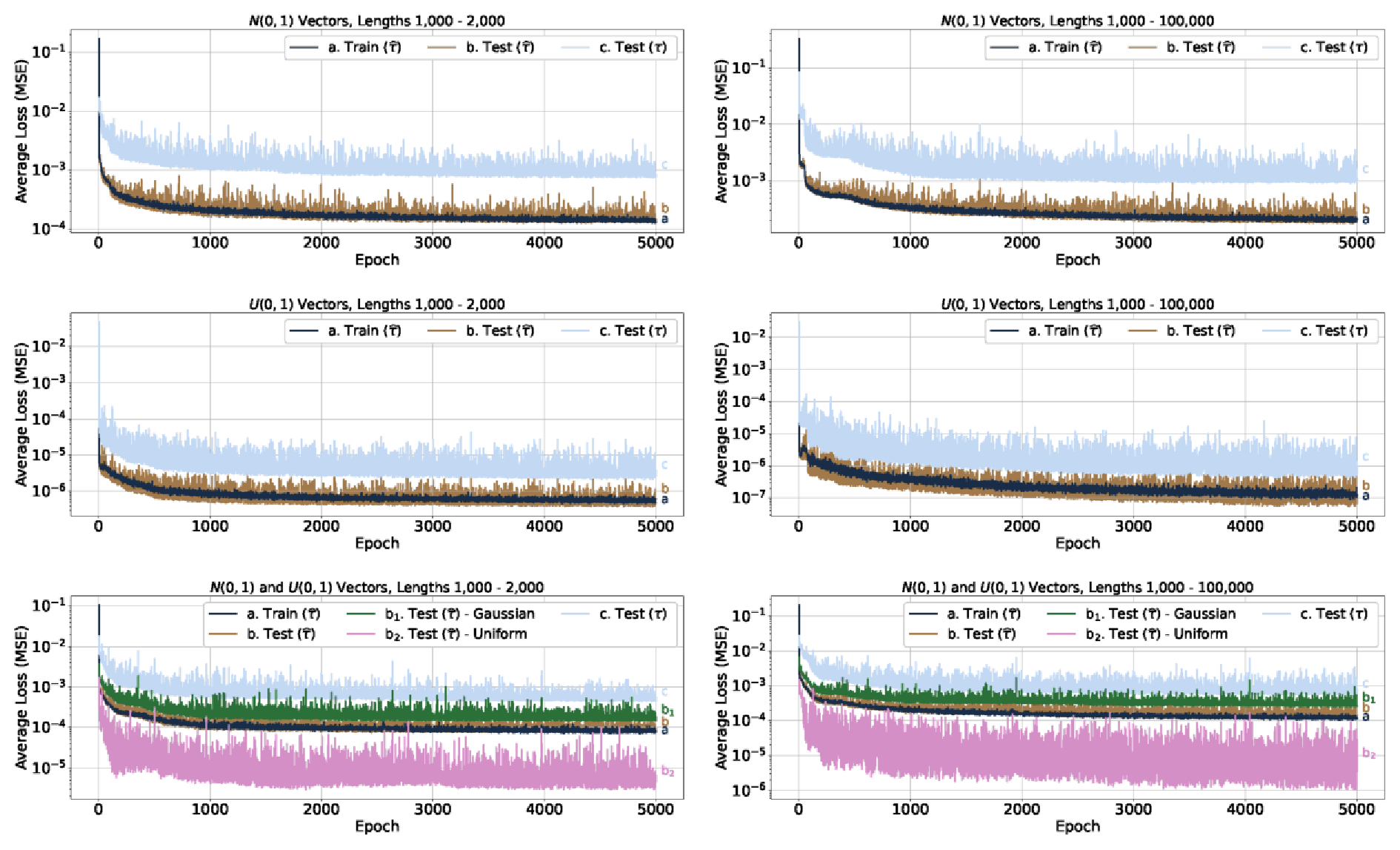}
    \caption{Neural network learning curves for experiments 1-6}
    \label{fig:fnn_lc}
\end{sidewaysfigure}

In all experiments the $\widehat{\tau}$ testing error is about one order of magnitude lower than that for $\tau$, which is to be expected since $\alpha \in [1,6)$.  In experiments 1-4, we see that the network performs better on $\mathcal{U}(0,1)$ vector data than $\mathcal{N}(0,1)$ vector data.  For vectors of length 1,000-2,000, the $\tau$ testing error for $\mathcal{U}(0,1)$ data is two orders of magnitude lower than that for $\mathcal{N}(0,1)$ data ($10^{-6}$ vs. $10^{-4}$), and for vectors of length 1,000-100,000 the difference is four orders of magnitude ($10^{-7}$ vs. $10^{-3}$).  This suggests that performance of the network is highly dependent on the vector data distribution.  See Appendix \ref{app:dens} for additional experiments that explore the difference in network performance for uniformly distributed and Gaussian vector data.  

We also note from experiments 3-4 that for $\mathcal{U}(0,1)$ vector data, a larger range of lengths in the dataset improves the $\tau$ testing error by an order of magnitude, i.e., $10^{-7}$ vs. $10^{-6}$.  For $\mathcal{N}(0,1)$ vector data, a larger range of lengths in the dataset does not significantly change the $\tau$ testing error.  These patterns generally hold for experiments 5-6 when the dataset includes both $\mathcal{N}(0,1)$ and $\mathcal{U}(0,1)$ vector data.  In addition, in experiments 5-6, the $\widehat{\tau}$ testing error is on the order of $10^{-4}$ for $\mathcal{N}(0,1)$ vector data, which is the same as in experiments 1-2.  For $\mathcal{U}(0,1)$ vector data in experiments 5-6, the $\widehat{\tau}$ testing error is on the order of $10^{-6}$, which is one to two orders of magnitude higher as compared to experiments 3-4, i.e., $10^{-6}$ vs. $10^{-7}$, for vector lengths 1,000-2,000 and $10^{-6}$ vs. $10^{-8}$, for vector lengths 1,000-100,000. 

To gauge the performance of the network, we compare it with a ``vanilla neural network'' that naively approximates $\tau$ using a neural network.  The vanilla network processes input vectors of different lengths by zero-padding the end of each vector until all vectors in the dataset are the same length.  By Theorem \ref{thm:zero_pad}, this does not affect $\tau$.

\begin{restatable}{theorem}{zeroPad}
\label{thm:zero_pad}
    Suppose $\mathbf{x} \in \mathbb{R}^{m}$ and $\tau = || \textbf{prox}_{\alpha ||\cdot||_\infty}(\mathbf{x}) ||_\infty$. If $\mathbf{x}_{\text{zero-pad}}$ is a zero padded version of $\mathbf{x}$, i.e., $\mathbf{x}_{\text{zero-pad}} = (\mathbf{x}_1, \mathbf{x}_2, ..., \mathbf{x}_m, 0, ..., 0)$, then $\tau_{\text{zero-pad}} = || \textbf{prox}_{\alpha ||\cdot||_\infty}(\mathbf{x}_{\text{zero-pad}}) ||_\infty = \tau$. 
\end{restatable}

\begin{proof}
    See Appendix \ref{app:zero_pad_pf}.
\end{proof}

Vanilla network preprocessing is given in Algorithm \ref{alg:vanilla_pre}, where $\ell$ is the maximum vector length in the dataset. The vanilla network architecture is given in Table \ref{tab:vnn_arch}. 
The architecture was chosen to create a fair comparison to the features based network presented previously.    

\begin{algorithm}
\caption{Data preprocessing for vanilla neural network}
\begin{algorithmic}[1]
\Require $\mathbf{x} \in \mathbb{R}^{m}$, $\alpha \geq 0$, $\tau$,  $\ell$
\Ensure $\widehat{\mathbf{x}} \in \mathbb{R}^{m}$, $\widehat{\tau} \in \mathbb{R}$
\If {$\text{length}(\mathbf{x}) < \ell$}
    \State Zero-pad $\mathbf{x}$ so that $\text{length}(\mathbf{x}) = \ell$
\EndIf
\State $\widehat{\mathbf{x}} = |\mathbf{x}|/\alpha$
\If {$\| \widehat{\mathbf{x}} \|_1 \leq 1$}
    \State $\widehat{\mathbf{x}} = \emptyset$, $\widehat{\tau} = \emptyset$  
    \Comment An empty vector and scalar are returned.
\Else
   \State $\widehat{\tau} = \tau/\alpha$
\EndIf
\State \Return $\widehat{\mathbf{x}}$, $\widehat{\tau}$ 
\end{algorithmic}
\label{alg:vanilla_pre}
\end{algorithm}

\begin{table}[h]
    \caption{Vanilla neural network architecture for experiments V1-V6.  $\ell$ is the maximum vector length in the dataset}
    \label{tab:vnn_arch}
    \begin{tabular*}{\textwidth}{l c c c c c}
        \toprule
        \textbf{Layer} & \textbf{Input} & \textbf{1st Hidden} & \textbf{2nd Hidden} & \textbf{3rd Hidden} & \textbf{Output} \\
        \midrule
        \textbf{Number of Neurons} & $\ell$ & 200 & 100 & 50 & 1 \\
        \textbf{Activation Function} & - & ReLU & ReLU & ReLU & - \\
        \botrule
    \end{tabular*}
\end{table}

The vanilla network experiments are structured in the same way and use the same data as the features based network experiments.  The only exceptions are that 1) due to the network size, the training and testing of the network were completed on an NVIDIA A100 SXM4 40GB GPU, and 2) the testing error on $\tau$ (line c) is calculated using $\alpha \widetilde{\tau}$ as the approximation to $\tau$ for an $\mathbf{x}$-$\alpha$-$\tau$ triple.  We will denote these experiments V1, V2, ..., V6 to avoid confusion with experiments 1-6 on the features based network.  The learning curves for experiments V1-V6 are presented in Figure \ref{fig:vnn_lc}.  

\begin{sidewaysfigure}
    \centering
    \includegraphics[width=1.0\textwidth]{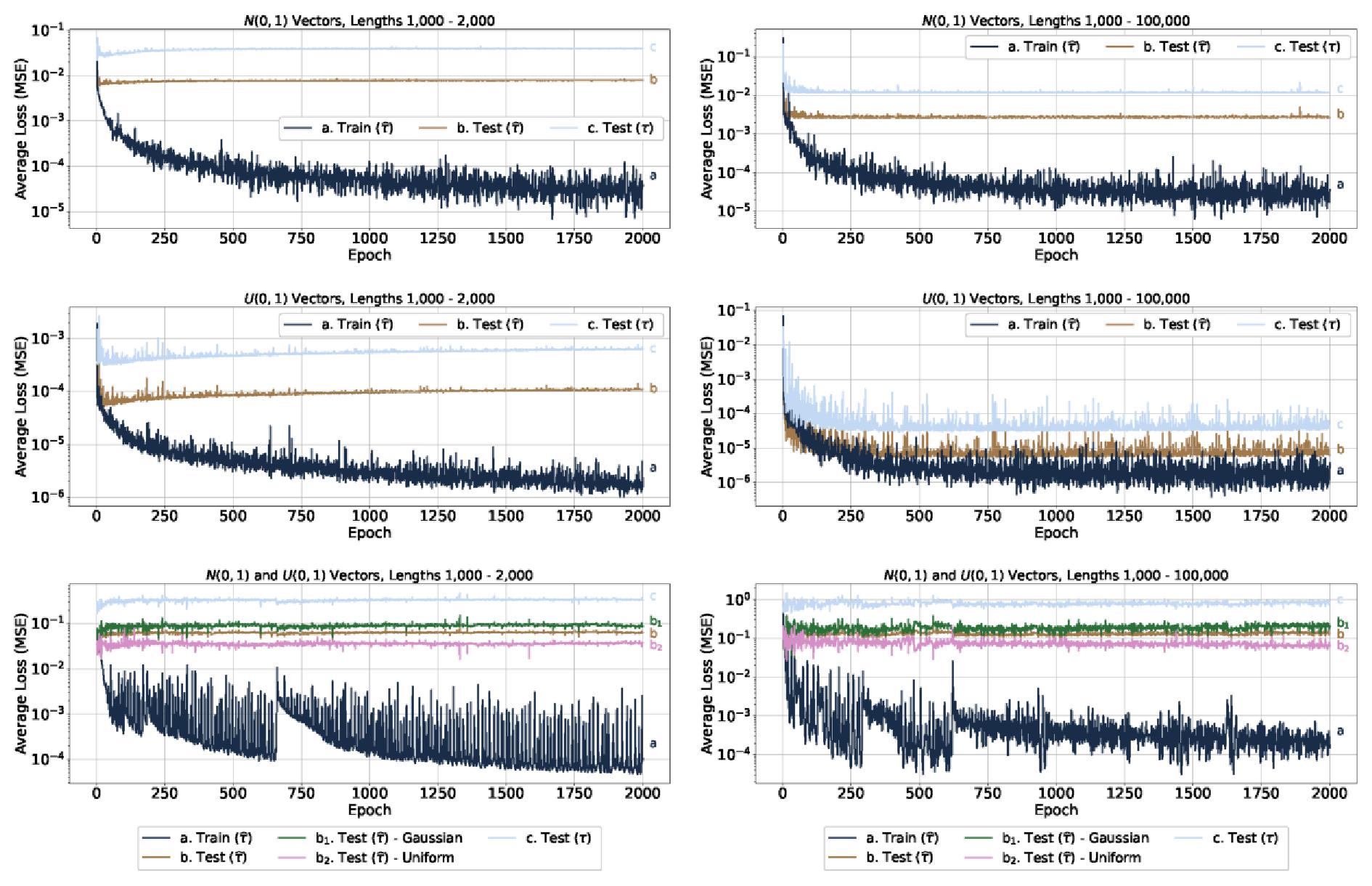}
    \caption{Vanilla neural network learning curves for experiments V1-V6. $\ell=$ 2,000 and $\ell=$ 100,000 for the experiments in the left and right columns, respectively}
    \label{fig:vnn_lc}
\end{sidewaysfigure}

The features based network outperforms the vanilla network.  For $\mathcal{N}(0,1)$ vector data in experiments V1-V2, the $\tau$ testing error is one to two orders of magnitude higher than the features based network ($10^{-2}$ vs. $10^{-4}$ for vector lengths 1,000-2,000 and $10^{-2}$ vs. $10^{-3}$ for vector lengths 1,000-100,000). For $\mathcal{U}(0,1)$ vector data in experiments V3-V4, it is two orders of magnitude higher ($10^{-4}$ vs. $10^{-6}$ for vector lengths 1,000-2,000 and $10^{-5}$ vs. $10^{-7}$ for vector lengths 1,000-100,000).  Similar to the features based network, the vanilla network performs better on $\mathcal{U}(0,1)$ vector data than $\mathcal{N}(0,1)$ vector data. However, in the vanilla network, we see overfitting in experiments V1 and V3 on vectors with a smaller range of lengths. Overfitting was not encountered with the features based network.  

In experiments V5-V6, the vanilla network struggles to learn when both $\mathcal{N}(0,1)$ and $\mathcal{U}(0,1)$ vector data are included in the dataset.  The $\widehat{\tau}$ testing error is one to two orders of magnitude higher for $\mathcal{N}(0,1)$ vector data as compared to experiments V1-V2 ($10^{-2}$ vs. $10^{-3}$ for vector lengths 1,000-2,000 and $10^{-1}$ vs. $10^{-3}$ for vector lengths 1,000-100,000), and three to four orders of magnitude higher for $\mathcal{U}(0,1)$ vector data as compared to experiments V3-V4 ($10^{-2}$ to $10^{-5}$ for vector lengths 1,000-2,000 and $10^{-2}$ to $10^{-6}$ for vector lengths 1,000-100,000).  

\subsubsection{Proximal Operator Error}

In each of experiments 1-6, we saved the features based network that achieved the lowest $\tau$ testing error.  We call these the best models for experiments 1-6.  For each best model, we compute the error in the proximal operator and the objective function (defined below) for each instance in the testing set.  For $\mathbf{x} \in \mathbb{R}^m$, let $\mathbf{p(x)}$ be the exact proximal operator of $\mathbf{x}$, i.e., $[\mathbf{p(x)}]_k = \sigma_{\tau}(\mathbf{x}_k)$, and let $\widetilde{\mathbf{p}}(\mathbf{x})$ be the approximate proximal operator of $\mathbf{x}$ computed using $\widetilde{\tau}$, the output of the features based network, i.e., $[\widetilde{\mathbf{p}}(\mathbf{x})]_k = \sigma_{\alpha(\widetilde{\tau} + \mu)}(\mathbf{x}_k)$. Given $\mathbf{x, y} \in \mathbb{R}^m$, let the objective function $f$ be defined as $f(\mathbf{x}, \mathbf{y}) = \frac{1}{2}||\mathbf{y-x}||_2^2 + \alpha ||\mathbf{y}||_\infty$. Define the proximal operator error, $\delta_{\mathbf{p}}(\mathbf{x})$,  and the objective function error, $\delta_{f}(\mathbf{x})$, as  
\[
    \delta_{\mathbf{p}}(\mathbf{x}) = \frac{\|\mathbf{p}(\mathbf{x})-\widetilde{\mathbf{p}}(\mathbf{x}) \|_2}{\| \mathbf{p}(\mathbf{x}) \|_2},
\] 
and
\[
    \delta_{f}(\mathbf{x}) = \frac{f(\mathbf{x}, \widetilde{\mathbf{p}}(\mathbf{x})) - f(\mathbf{x}, \mathbf{p(x)})}{f(\mathbf{x}, \mathbf{p(x)})}. 
\]
In Table \ref{tab:fnn_err}, we report the median, average, and standard deviation for the proximal operator and objective function errors.  As expected from the learning curves for experiments 1-6, lower median and average $\delta_{\mathbf{p}}$ and $\delta_{f}$ values are achieved for $\mathcal{U}(0,1)$ vector data than $\mathcal{N}(0,1)$ vector data, usually by an order of magnitude, for vectors of lengths 1,000-2,000 and 1,000-100,000.  In addition, median and average $\delta_{\mathbf{p}}$ values for vectors of lengths 1,000-100,000 are less than those for the corresponding vectors of lengths 1,000-2,000. Median and average $\delta_{\mathbf{p}}$ and $\delta_{f}$ values for experiments 5-6 with $\mathcal{N}(0,1)$ and $\mathcal{U}(0,1)$ vector data are higher than those of only $\mathcal{N}(0,1)$ or $\mathcal{U}(0,1)$ vector data in experiments 1-2 and 3-4, respectively.

\begin{table}[h]
    \caption{Median (mdn), average (avg), and standard deviation (sd) of $\delta_{\mathbf{p}}$ and $\delta_{f}$ over the testing set for the best models from experiments 1-6}
    \label{tab:fnn_err}
    \begin{tabular*}{\textwidth}{@{\extracolsep\fill}l@{\hskip0.5\tabcolsep}l@{\hskip0.5\tabcolsep}lcccccc}
        \toprule%
        & & & \multicolumn{3}{@{}c@{}}{$\delta_{\mathbf{p}}$} & \multicolumn{3}{@{}c@{}}{$\delta_{f}$} \\
        \cmidrule{4-6} \cmidrule{7-9}%
        \textbf{Ex} & \textbf{Data} & \textbf{Vector Len} & \textbf{mdn} & \textbf{avg} & \textbf{sd} & \textbf{mdn} & \textbf{avg} & \textbf{sd} \\
        \midrule
        1 & $\mathcal{N}(0,1)$ & 1,000 - 2,000 & 1.5e-3 & 1.8e-3 & 1.5e-3 & 1.7e-4 & 4.1e-4 & 6.8e-4  \\
        2 & $\mathcal{N}(0,1)$ & 1,000 - 100,000 & 3.7e-4 & 5.4e-4 & 7.5e-4 & 2.4e-4 & 5.6e-4 & 9.8e-4  \\
        3 & $\mathcal{U}(0,1)$ & 1,000 - 2,000 & 4.8e-4 & 5.8e-4 & 4.7e-4 & 1.8e-5 & 4.4e-5 & 7.3e-5  \\
        4 & $\mathcal{U}(0,1)$ & 1,000 - 100,000 & 7.1e-5 & 1.2e-4 & 1.9e-4 & 1.3e-5 & 2.9e-5 & 4.6e-5  \\
        5 & $\mathcal{N}(0,1)$ \& $\mathcal{U}(0,1)$ & 1,000 - 2,000 & 1.6e-3 & 1.9e-3 & 1.5e-3 & 2.0e-4 & 4.5e-4 & 7.1e-4  \\
        6 & $\mathcal{N}(0,1)$ \& $\mathcal{U}(0,1)$ & 1,000 - 100,000 & 7.5e-4 & 1.1e-3 & 1.3e-3 & 2.7e-3  & 4.1e-3 & 4.1e-3  \\
        \botrule
    \end{tabular*}
\end{table}

\subsubsection{Feature Importance}

We used the best models from experiments 1-6 to measure feature importance by computing saliency for each feature, i.e., the gradient of the network output with respect to each feature.  A feature with a higher saliency is considered more important to the network, since a small change in the value of the feature will result in a larger change in the output of the network. 
 For each best model we computed the average saliency of each feature using the testing set feature vectors, i.e., the $\mathbf{w}$ vectors that resulted from inputting the testing set data vectors into Algorithm \ref{alg:features_pre}.  The results are presented in Figure \ref{fig:fnn_fi}.  

\begin{figure}[h]
    \centering
    \includegraphics[width=1.0\textwidth]{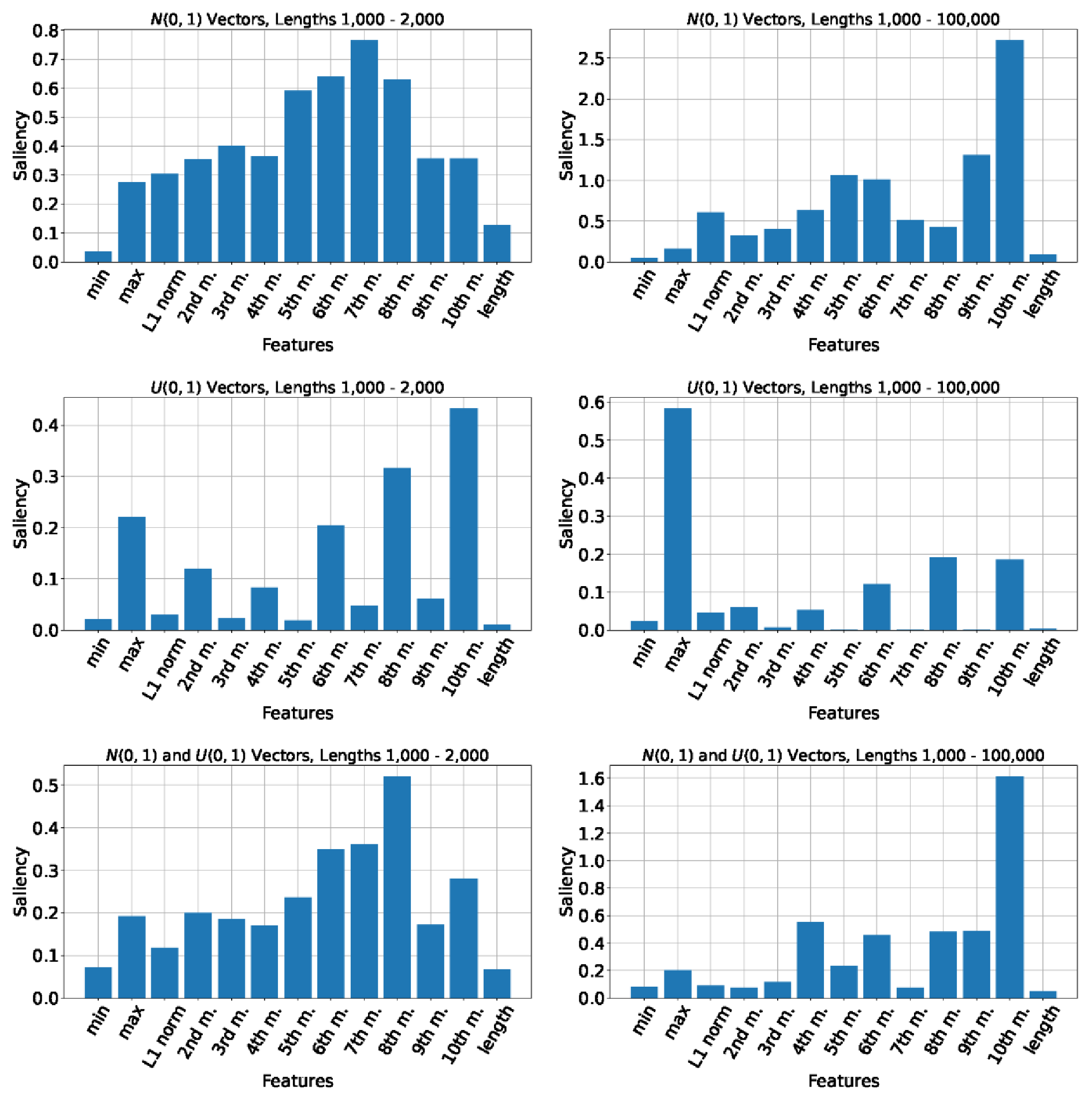}
    \caption{Feature importances for the best models from experiments 1-6}
    \label{fig:fnn_fi}
\end{figure}

For $\mathcal{N}(0,1)$ vectors of lengths 1,000-2,000, the 5th-8th moments are relatively more important than other features; however, for $\mathcal{N}(0,1)$ vectors of lengths 1,000-100,000, the 10th moment is by far the most important feature.  For $\mathcal{N}(0,1)$ vector data of both length ranges, the minimum is relatively unimportant.  For $\mathcal{U}(0,1)$ vectors of lengths 1,000-2,000, the 10th moment is the most important feature, but the maximum, 6th and 8th moments are also relatively important.  This is in contrast to $\mathcal{U}(0,1)$ vectors of lengths 1,000-100,000 in which the maximum is by far the most important feature.  For $\mathcal{U}(0,1)$ vector data of both length ranges, the even moments are more important than the odd moments, the minimum, and the log of the length.   

Experiments 5-6 with $\mathcal{N}(0,1)$ and $\mathcal{U}(0,1)$ vector data exhibit patterns that are fairly similar to those in experiments 1-2 with $\mathcal{N}(0,1)$ vector data only.  However, the 8th moment is the most important feature for $\mathcal{N}(0,1)$ and $\mathcal{U}(0,1)$ vectors with lengths 1,000-2,000, as opposed to the 7th moment for $\mathcal{N}(0,1)$ vectors of lengths 1,000-2,000.

\subsubsection{Computational Efficiency}

Table \ref{tab:fnn_times} gives the average times to compute 1) an approximate proximal operator using the features based network and 2) an exact proximal operator using Algorithm \ref{alg:base_min}.  The features based network approach is divided into preprocessing, which includes feature generation, i.e., Algorithm \ref{alg:features_pre} without Line 16 (Line 16 computes a scaled, centered version of $\tau$); inference; and proximal operator computation time (given $\widetilde{\tau}$).  The times are averaged over 10,000 vectors of the given length. The lowest average time between the approximate and exact proximal operator is given in bold.  For each vector length in Table \ref{tab:fnn_times}, we used 5,000 $\mathcal{N}(0,1)$ vectors and 5,000 $\mathcal{U}(0,1)$ vectors of the given length. The best model from experiment 5 was used for inference for vectors of length 1,000 and the best model from experiment 6 was used for inference for vectors of lengths 10,000 and 100,000.  All calculations are double precision except the inference step of the features based network approach, which is single precision.  The preprocessing and proximal operator computations for the features based network approach and the exact proximal operator computations operate on one vector at a time via a loop.  Since these are embarrassingly parallel computations, we used the numba package and jit function decorator in Python as a speed up.  

\begin{table}[h]
    \caption{Average times (seconds) to compute 1) an approximate proximal operator using the features based network and 2) an exact proximal operator using Algorithm \ref{alg:base_min}.  Averages are computed over 10,000 vectors of the given length}
    \label{tab:fnn_times}
    \begin{tabular*}{\textwidth}{@{\extracolsep\fill}lccccc}
        \toprule%
         & \multicolumn{4}{@{}c@{}}{\textbf{Features Based Network Avg Time}}  & \multicolumn{1}{@{}c@{}}{\textbf{Exact Avg Time}} \\ 
        \cmidrule{2-5} \cmidrule{6-6}%
         \textbf{Vector Len} & \textbf{Preprocessing} & \textbf{NN Inf} & \textbf{Prox Op} & \textbf{Total} & \textbf{Total} \\
        \midrule
        1,000 & 2.6e-6 & 4.5e-6 & 1.2e-6 & 8.2e-6 & \textbf{7.8e-6} \\
        10,000 & 2.7e-5 & 4.6e-6 & 8.9e-6 & \textbf{4.1e-5} & 1.0e-4 \\
        100,000 & 2.6e-4 & 9.3e-6 & 1.2e-4 & \textbf{3.9e-4} & 1.4e-3 \\
        \botrule
    \end{tabular*}
\end{table}

As vector length increases, the times for the features based network and exact approaches increase, and the approximate proximal operator is computed faster than the exact proximal operator by an order of magnitude.  Inference time is similar for all vector lengths since the network depends on the number of moments calculated for each vector, not vector length.  The results in Table \ref{tab:fnn_times} are consistent with the fact that for $\mathbf{x} \in \mathbb{R}^m$, the features based network approach has complexity $O(m)$ and the exact approach in Algorithm \ref{alg:base_min} has complexity $O(m \log m)$.  For the features based network approach, the constant for the $O(m)$ complexity depends on the number of moments chosen to compute, which motivates selecting a small value for the number of moments.

\section{Conclusion}
\label{sec:conc}

We theoretically developed an approach to compute $\textbf{prox}_{\alpha ||\cdot||_\infty}(\mathbf{x})$, presented the corresponding $O(m \log m)$ algorithm, and related it to the approach to compute $\textbf{prox}_{\alpha ||\cdot||_\infty}(\mathbf{x})$ that uses the Moreau decomposition.  We then developed a neural network to approximate $\tau$, the threshold calculated in the computation of $\textbf{prox}_{\alpha ||\cdot||_\infty}(\mathbf{x})$, using the exact approach to compute $\textbf{prox}_{\alpha ||\cdot||_\infty}(\mathbf{x})$ as motivation for data preprocessing.  The network is novel in that it can accept vectors of varying lengths due to a feature selection process based on the moments of the input data.  We showed through numerical experiments that the network can efficiently learn $\tau$, and does so with better performance than a vanilla network that naively approximates $\tau$ without using a feature selection process.  Hence, we can use the network to approximate $\textbf{prox}_{\alpha ||\cdot||_\infty}(\mathbf{x})$ in $O(m)$ complexity.  Two potential areas for future work are 1) further investigation into the difference of network performance for various vector data distributions, and 2) performance assessment of using the network approximation of the proximal operator in applications as compared to using the exact proximal operator.  

\backmatter

\bmhead{Acknowledgements}
The authors acknowledge Research Computing at The University of Virginia for providing computational resources that have contributed to the results reported within this paper (URL: \url{https://rc.virginia.edu}).  

\bmhead{Funding} The second author has been supported in part by the National Science Foundation under grant DMS-2108900 and by the Simons Foundation.

\bmhead{Data Availability Statement} The data generated for experiments in this manuscript can be reproduced using code available at \url{https://github.com/klinehan1/prox_op_nn}. 

\section*{Declarations}

\bmhead{Conflict of interest} On behalf of all authors, the corresponding author states that there is no conflict of interest.

\bmhead{Code availability} Experiment code is available at \url{https://github.com/klinehan1/prox_op_nn}. 

\begin{appendices}
\setcounter{figure}{3}
\setcounter{table}{5}
\renewcommand\thefigure{\arabic{figure}}
\renewcommand\thetable{\arabic{table}}

\section{Proof of Theorem \ref{thm:psi}}
\label{app:psi_pf}

\psiThm*

We begin by providing a lemma.

\begin{lemma}
\label{lemma:psi_cont}
    Suppose the assumptions from Theorem \ref{thm:psi} are true.  Then $\psi$ is continuous over $t \geq 0$.
\end{lemma}
\begin{proof}
     We first show that $\psi$ is piecewise continuous. Suppose $t > \mathbf{s}_1$.  Then $\psi(t) = \alpha t$ is a linear function.  In addition, for each index $i \in [m]$ such that $\mathbf{s}_{i+1} < t < \mathbf{s}_{i}$, $|I_t| = i$ and $\psi$ is a quadratic.  Hence $\psi$ is piecewise continuous.  Now, suppose $t = \mathbf{s}_1$. Then 
    \begin{equation*}
        \lim_{t \rightarrow \mathbf{s}_1^+} \psi(t) = \lim_{t \rightarrow \mathbf{s}_1^+} \left[ \alpha t \right] = \psi(\mathbf{s}_1) 
        = \lim_{t \rightarrow \mathbf{s}_1^-} \left[ \frac{1}{2}\sum_{\substack{k: \\ \mathbf{s}_k = \mathbf{s}_1}}(\mathbf{s}_k - t)^2 + \alpha t \right] = \lim_{t \rightarrow \mathbf{s}_1^-} \psi(t).
    \end{equation*}
    Similarly for $t = \mathbf{s}_{m+1} = 0$, 
    \[
         \lim_{t \rightarrow \mathbf{s}_{m+1}^+} \psi(t) = \lim_{t \rightarrow \mathbf{s}_{m+1}^+} \left[ \frac{1}{2}\sum_{k=1}^{m} (\mathbf{s}_k - t)^2 + \alpha t \right] = \psi(\mathbf{s}_{m+1}). 
    \]
    Next, suppose $t = \mathbf{s}_i$ for some $i \in [m]\setminus \{1\}$.  Let $S_i = \{k:\mathbf{s}_k = \mathbf{s}_i\}$.  Then
    \begin{equation*}
        \begin{split}
            \lim_{t \rightarrow \mathbf{s}_i^+} \psi(t) & = \lim_{t \rightarrow \mathbf{s}_i^+}\left[ \frac{1}{2} \sum_{k \in [i] \setminus S_i} (\mathbf{s}_k - t)^2  + \alpha t \right] = \psi(\mathbf{s}_i) \\ 
            & = 
            \lim_{t \rightarrow \mathbf{s}_i^-}\left[ \frac{1}{2}\sum_{k \in [i] \cup S_i} (\mathbf{s}_k - t)^2 + \alpha t \right]= \lim_{t \rightarrow \mathbf{s}_i^-} \psi(t).
        \end{split}
    \end{equation*}
    Thus, $\psi$ is continuous over $t \geq 0$.
\end{proof}

\begin{proof}[Proof of Theorem \ref{thm:psi}] 
We will prove the four results from the theorem separately. \newline

\noindent \textit{$\psi$ is $\mathcal{C}^1$ over $t > 0$:} By Lemma \ref{lemma:psi_cont}, $\psi$ is continuous, linear over $t > \mathbf{s}_1$, and quadratic over $\mathbf{s}_{k+1} < t < \mathbf{s}_{k}$ for $k \in [m]$ such that $\mathbf{s}_{k+1} \neq \mathbf{s}_{k}$.  Hence $\psi$ is $\mathcal{C}^1$ on $t > \mathbf{s}_1$, and $\mathbf{s}_{k+1} < t < \mathbf{s}_{k}$  for $k \in [m]$ such that $\mathbf{s}_{k+1} \neq \mathbf{s}_{k}$.  We now turn our attention to the points $t = \mathbf{s}_k$ for $k \in [m]$.
    Suppose $t = \mathbf{s}_1$. Then 
    \begin{equation*}
        \lim_{t \rightarrow \mathbf{s}_1^+} \psi'(t)  = \lim_{t \rightarrow \mathbf{s}_1^+} \alpha = \alpha = 
           \lim_{t \rightarrow \mathbf{s}_1^-}\left[-\sum_{\substack{k: \\ \mathbf{s}_k = \mathbf{s}_1}} (\mathbf{s}_k - t) + \alpha \right] =  \lim_{t \rightarrow \mathbf{s}_1^-} \psi'(t),
    \end{equation*}
    and thus $\psi'(\mathbf{s}_1) = \alpha$.
    Next, fix an index $i \in [m]\setminus \{1\}$ and suppose $t = \mathbf{s}_i$.   Let $S_i = \{k:\mathbf{s}_k = \mathbf{s}_i\}$. Then
    \begin{equation*}
        \lim_{t \rightarrow \mathbf{s}_i^+} \psi'(t)  = \lim_{t \rightarrow \mathbf{s}_i^+}\left[-\sum_{k \in [i] \setminus S_i} (\mathbf{s}_k - t) + \alpha \right] = 
           \lim_{t \rightarrow \mathbf{s}_i^-}\left[ -\sum_{k \in [i] \cup S_i} (\mathbf{s}_k - t) + \alpha \right] =  
           \lim_{t \rightarrow \mathbf{s}_i^-} \psi'(t), 
    \end{equation*}
    and thus $\psi'(\mathbf{s}_i) = -\sum_{k \in [i] \setminus S_i} (\mathbf{s}_k - \mathbf{s}_i) + \alpha$.  Hence $\psi$ is $\mathcal{C}^1$ over $t>0$. \newline

    \noindent \textit{$\psi$ is strictly convex over $0 \leq t \leq ||\mathbf{x}||_\infty = \mathbf{s}_1$:} We will show that $\psi'$ is strictly increasing over $0 < t < \mathbf{s}_1$ thus proving the result \cite[Theorems 12A, 12B]{roberts_varberg_1973}.  Let $0 < t_1 < t_2 < \mathbf{s}_1$. This guarantees that $|I_{t_1}| >0$ and $|I_{t_2} | > 0$. We proceed using two cases. \newline 
    Case 1: Suppose $t_1$ and $t_2$ are such that $I_{t_2} = I_{t_1}$.  Then 
    \[
        \psi'(t_1) = |I_{t_1}|t_1 - \sum_{k \in I_{t_1}} |\mathbf{x}_k | + \alpha < |I_{t_2}| t_2 - \sum_{k \in I_{t_2}} |\mathbf{x}_k | + \alpha = \psi'(t_2),
    \]   
    hence $\psi'$ is strictly increasing. \newline
    Case 2: Suppose $t_1$ and $t_2$ are such that $I_{t_2} \subset I_{t_1}$.  Then $t_1 \leq |\mathbf{x}_k|$ for $k \in I_{t_1} - I_{t_2}$ and 
    \[
        |I_{t_1} - I_{t_2}|t_1 - \sum_{k \in I_{t_1} - I_{t_2}} |\mathbf{x}_k | \leq 0.
    \]
    Since
    \begin{equation*}
        \begin{split}
             |I_{t_2}| t_1 & < |I_{t_2}| t_2  \Rightarrow 
             |I_{t_2}| t_1 + |I_{t_1} - I_{t_2}|t_1 - \sum_{k \in I_{t_1} - I_{t_2}} |\mathbf{x}_k | < |I_{t_2}| t_2  \\ & \Rightarrow 
             (|I_{t_2}| + |I_{t_1} - I_{t_2}|)t_1 - \left( \sum_{k \in I_{t_2}} |\mathbf{x}_k | + \sum_{k \in I_{t_1} - I_{t_2}} |\mathbf{x}_k | \right) + \alpha < |I_{t_2}| t_2 - \sum_{k \in I_{t_2}} |\mathbf{x}_k | + \alpha \\ & \Rightarrow 
             |I_{t_1}|t_1 - \sum_{k \in I_{t_1}} |\mathbf{x}_k | + \alpha < |I_{t_2}| t_2 - \sum_{k \in I_{t_2}} |\mathbf{x}_k | + \alpha \\ 
             & \Rightarrow 
             \psi'(t_1) < \psi'(t_2),
        \end{split}
    \end{equation*}
    $\psi'$ is strictly increasing. \newline

    \noindent \textit{$\tau \leq ||\mathbf{x}||_\infty = \mathbf{s}_1$:} Since $\textbf{prox}_{\alpha ||\cdot||_\infty}(\mathbf{0}) = \mathbf{0}$, it is clear that if $\mathbf{x} = \mathbf{0}$, then $\tau = 0 = \mathbf{s}_1$.  
    If $\mathbf{x} \neq \mathbf{0}$ and $\alpha>0$, then $\psi'(t) = \alpha > 0$ for $t \geq \mathbf{s}_1$.  Thus $\tau < \mathbf{s}_1 = ||\mathbf{x}||_\infty$.  If $\mathbf{x} \neq \mathbf{0}$ and $\alpha = 0$, it is clear that $\textbf{prox}_{\alpha ||\cdot||_\infty}(\mathbf{x}) = \mathbf{x}$ and the smallest $\tau$ that would guarantee this is $\tau = \mathbf{s}_1$.  \newline

    \noindent \textit{$\tau = 0$ if and only if $|| \mathbf{x} ||_1 \leq \alpha$:} We will prove that if $\tau = 0$, then $|| \mathbf{x} ||_1 \leq \alpha$ by contradiction.  Suppose $\tau=0$ and $|| \mathbf{x} ||_1 > \alpha$.  Then $\psi'(0) = -\| \mathbf{x} \|_1 + \alpha < 0$.  In addition, $\psi'(\mathbf{s}_1) = \alpha \geq 0$.  So, by the fact that $\psi$ is $\mathcal{C}^1$ over $t > 0$ and the Intermediate Value Theorem, there exists a $0 < t \leq \mathbf{s}_1$ such that $\psi'(t)=0$. Since $\psi$ is strictly convex over $0 \leq t \leq ||\mathbf{x}||_\infty = \mathbf{s}_1$, this contradicts $\tau=0$.  \newline 
    Now suppose $|| \mathbf{x} ||_1 \leq \alpha$.  Then $\psi'(0) = -\| \mathbf{x} \|_1 + \alpha \geq 0$.  Since $\psi$ is strictly convex over $0 \leq t \leq ||\mathbf{x}||_\infty = \mathbf{s}_1$, $\tau=0$.   
    
\end{proof}

\section{Divide and Conquer Algorithm}
\label{app:dc_alg}

\begin{algorithm}
\caption{Linear time computation of the proximal operator of the $\ell_\infty$ norm}
\begin{algorithmic}[1]
\Require $\mathbf{x} \in \mathbb{R}^{m}$, $\alpha \geq 0$
\Ensure $\textbf{prox}_{\alpha ||\cdot||_\infty}(\mathbf{x}) \in \mathbb{R}^{m}$ 
\If {$||\mathbf{x}||_1 \leq \alpha$}
    \State $\textbf{prox}_{\alpha ||\cdot||_\infty}(\mathbf{x}) = \mathbf{0}$
\Else
    \State $\mathbf{\widehat{x}} = \{ |\mathbf{x}_{i}| \: |\:  \mathbf{x}_{i} \neq 0\}$ 
    \State $\boldsymbol\ell = \emptyset$, $\mathbf{u} = \emptyset$, $\boldsymbol\ell^{\text{max}} = 0$
    \State $\widetilde{\nu} = 0$, $n_{\widetilde{\nu}} = 0$  
    \While{$\text{length}(\mathbf{\widehat{x}}) \, \neq 0$} 
        \Comment Find $\boldsymbol\ell^{\text{max}}$ and $\mathbf{u}^{\text{min}}$
        \State Randomly select a pivot, $p$, from the elements of $\mathbf{\widehat{x}}$.
        \State \parbox[t]{\dimexpr\textwidth-\leftmargin-\labelsep-\labelwidth}{Partition $\mathbf{\widehat{x}}$ such that $\boldsymbol\ell = \{ \mathbf{\widehat{x}}_{i}|\:  \mathbf{\widehat{x}}_{i} < p\}$ and $\mathbf{u} = \{ \mathbf{\widehat{x}}_{i}|\:  \mathbf{\widehat{x}}_{i} > p\}$. 
 Let $n_{\nu} = \text{length}(\mathbf{u})$, $\nu = \sum_{k=1}^{n_{\nu}}\mathbf{u}_k$, and $n_p = \text{length}(\{ \mathbf{\widehat{x}}_{i}|\: \mathbf{\widehat{x}}_{i} = p\})$.\strut}
        \State $\psi'(p) = -(\widetilde{\nu}  + \nu) + (n_{\widetilde{\nu}} + n_\nu)p + \alpha$
        \Comment $\psi'(p) = -\sum_{k \in I_p} |\mathbf{x}_k| +  |I_{p}|p + \alpha$
        \If{$\psi'(p) < 0$}
            \State $\boldsymbol\ell^{\text{max}} = p$
            \State $\mathbf{\widehat{x}} = \mathbf{u}$
        \Else 
            \State $\mathbf{u}^{\text{min}} = p$
            \State $\mathbf{\widehat{x}} = \boldsymbol\ell$
            \State $\widetilde{\nu} = \widetilde{\nu} + \nu +pn_p$
            \State $n_{\widetilde{\nu}} = n_{\widetilde{\nu}} + n_{\nu} + n_p$
        \EndIf
    \EndWhile
    \State $t_0 = (\widetilde{\nu}-\alpha)/n_{\widetilde{\nu}} $
    \Comment Find the minimum of $\psi(t)$ 
    \If{$\boldsymbol\ell^{\text{max}} < t_0 \leq \mathbf{u}^{\text{min}}$}
        \State $\tau = t_0$
    \EndIf
    \State $\forall k \in [m]$ compute $[\textbf{prox}_{\alpha ||\cdot||_\infty}(\mathbf{x})]_k = \sigma_{\tau}(\mathbf{x}_k)$    
\EndIf
\State \Return $\textbf{prox}_{\alpha ||\cdot||_\infty}(\mathbf{x})$  
\end{algorithmic}
\label{alg:eff_min}
\end{algorithm}

Similar to the motivation for the expected linear time algorithm of \cite{duchi_etal_2008} for projection onto the simplex, the motivation for Algorithm \ref{alg:eff_min} is that we actually do not need a full sort of the input vector to find $\tau$; we only need to know the sum of the elements in the set $I_{\tau}$ and $|I_\tau|$. As in Algorithm \ref{alg:base_min}, the main iteration finds a $0 < t \leq \mathbf{s}_1$ such that $\psi'(t) = 0$.  Let $\mathbf{\widehat{x}} = \{ |\mathbf{x}_{i}| \: |\:  \mathbf{x}_{i} \neq 0\}$.  In each while loop iteration of Algorithm \ref{alg:eff_min} we choose an element of $\widehat{\mathbf{x}}$ (or a subset of $\widehat{\mathbf{x}}$) as the pivot, $p$, and partition $\widehat{\mathbf{x}}$ (or a subset of $\widehat{\mathbf{x}}$) by $p$.  We use the strict convexity of $\psi$ and the sign of $\psi'(p)$ to determine in which subset of $\widehat{\mathbf{x}}$ to continue searching for $\tau$. After execution of the while loop, there exists a $k \in [m]$ such that $\boldsymbol \ell^{\text{max}} = \mathbf{s}_{k+1}$ and $\mathbf{u}^{\text{min}} = \mathbf{s}_k$ for which $\boldsymbol \ell^{\text{max}} < \tau \leq \mathbf{u}^{\text{min}}$. At this point  we are able to calculate $\tau$ since the number of elements greater than or equal to $\mathbf{u}^{\min}$ and the sum of these elements have been updated in each iteration of the while loop.  Note that this is equivalent to how $\tau$ is computed in Algorithm \ref{alg:base_min}; the difference is that in Algorithm \ref{alg:base_min} we arrive at the interval $[\mathbf{s}_{k+1}, \mathbf{s}_k)$ by a linear search and in Algorithm \ref{alg:eff_min} we use a divide and conquer technique.  

The complexity of Algorithm \ref{alg:eff_min} is based on the pivot selected in line 8.  The expected complexity is $O(m)$, but the worst case complexity is $O(m^2)$.  We note that an $O(m)$ complexity can be achieved if the median of $\widehat{\mathbf{x}}$ is chosen as the pivot; however, this is not practical since finding the median of a vector in $O(m)$ time is slow \cite{condat_2016}. 

\section{Proof of Theorem \ref{thm:same_t}}
\label{app:samet_pf}

\sametThm*

We provide a helpful lemma before proving the theorem.

\begin{lemma}
\label{lem:repeat_t}
For $\alpha > 0$, let $\mathbf{x} \in \mathbb{R}^m$ be such that $||\mathbf{x}||_1 > \alpha$, and $\tau$ be calculated as in Algorithm \ref{alg:base_min}.  If $\tau = |\mathbf{x}_r| = \mathbf{s}_r$ and $\mathbf{s}_r$ appears $j$ times in $\mathbf{s}$, i.e., $\mathbf{s}_r = \mathbf{s}_{r-1} = \cdots = \mathbf{s}_{r-(j-1)}$, then 
\[
\tau = \left( \sum_{k=1}^{r} \mathbf{s}_k - \alpha \right)/r = \left( \sum_{k=1}^{r-1} \mathbf{s}_k - \alpha \right)/(r-1) = \cdots = \left( \sum_{k=1}^{r-j} \mathbf{s}_k - \alpha \right)/(r-j).
\]   
\end{lemma}

\begin{proof}
    By Algorithm \ref{alg:base_min}, $\tau = (\sum_{k=1}^{r} \mathbf{s}_k - \alpha)/r $.  To show the next part of the result, we use induction.  Let $i=1$.  Then,
    \[
        \mathbf{s}_r = \tau = \left( \sum_{k=1}^{r} \mathbf{s}_k - \alpha \right)/r \Leftrightarrow 
        \left( \sum_{k=1}^{r-1} \mathbf{s}_k - \alpha \right)/(r-1) = \mathbf{s}_r = \mathbf{s}_{r-1} = \tau.
    \]
    Now, assume that $\tau = (\sum_{k=1}^{r} \mathbf{s}_k - \alpha)/r = (\sum_{k=1}^{r-i} \mathbf{s}_k - \alpha)/(r-i)$ for some $i \in [j-1]$.  Then,
    \[
        \tau = \left( \sum_{k=1}^{r-i} \mathbf{s}_k - \alpha \right)/(r-i) = \mathbf{s}_{r-i} \Leftrightarrow 
        \left( \sum_{k=1}^{r-(i+1)} \mathbf{s}_k - \alpha \right)/(r-(i+1)) =  \mathbf{s}_{r-(i+1)} = \tau.
    \]
    Hence, by induction,
    \[
        \tau = \left( \sum_{k=1}^{r} \mathbf{s}_k - \alpha \right)/r = \left( \sum_{k=1}^{r-(j-1)} \mathbf{s}_k - \alpha \right)/(r-(j-1)) = \mathbf{s}_{r-(j-1)}, 
    \]
    and as a final step, 
    \[
        \tau = \left( \sum_{k=1}^{r-(j-1)} \mathbf{s}_k - \alpha \right)/(r-(j-1)) = \mathbf{s}_{r-(j-1)} \Leftrightarrow  \left( \sum_{k=1}^{r-j} \mathbf{s}_k - \alpha \right)/(r-j) = \mathbf{s}_{r-(j-1)} = \tau.
    \]
\end{proof}

\begin{proof}[Proof of Theorem \ref{thm:same_t}]  
  We will prove the result using two cases.  \newline 
  Case 1: Suppose $\mathbf{s}_{i+1} < \tau < \mathbf{s}_{i}$ for some $i \in [m]$ as found by Algorithm \ref{alg:base_min}.  Then, $\tau = (\sum_{k=1}^{i} \mathbf{s}_k - \alpha)/i$.  Hence $i^* \geq i$ for Algorithm \ref{alg:simp}.  We will show by induction that for all $j \in [m]$ such that $i < j \leq m$, $(\sum_{k=1}^{j} \mathbf{s}_k - \alpha)/j \geq \mathbf{s}_{j}$ and therefore $i^* = i$ and thus $\tau = \tau_{\text{simplex}}$.  For $j = i+1$, we have
    \[
         \left( \sum_{k=1}^{i} \mathbf{s}_k - \alpha \right)/i > \mathbf{s}_{i+1}  \Leftrightarrow  \left( \sum_{k=1}^{i+1} \mathbf{s}_k - \alpha \right) /(i+1) > \mathbf{s}_{i+1}.   
    \]
    Next, assume that $(\sum_{k=1}^{j} \mathbf{s}_k - \alpha)/j \geq \mathbf{s}_{j}$ for some $j \in [m]$ such that $i+1 \leq j \leq m$.  By the same argument above and the fact that $\mathbf{s}_j \geq \mathbf{s}_{j+1}$, we have $(\sum_{k=1}^{j+1} \mathbf{s}_k - \alpha)/(j+1) \geq \mathbf{s}_{j+1}$, and the result is proven. \newline 
    Case 2: Suppose $\tau = \mathbf{s}_i$ for some $i \in [m]$ as found by Algorithm \ref{alg:base_min}.  Then $\mathbf{s}_{i+1} < \tau \leq \mathbf{s}_{i}$ and $\tau = (\sum_{k=1}^{i} \mathbf{s}_k - \alpha)/i = \mathbf{s}_i$.  Let $\hat{i} < i$ be the largest index such that $\mathbf{s}_{\hat{i}} > \mathbf{s}_i$.  Then by Lemma \ref{lem:repeat_t}, 
  \[
    \mathbf{s}_{\hat{i}} > \mathbf{s}_i = \tau = \left( \sum_{k=1}^{i} \mathbf{s}_k - \alpha \right)/i = \left( \sum_{k=1}^{\hat{i}} \mathbf{s}_k - \alpha \right)/\hat{i}.
  \]
    By an argument similar to that of Case 1, for all $j \in [m]$ such that $\hat{i} < j \leq m$, $(\sum_{k=1}^{j} \mathbf{s}_k - \alpha)/j \geq \mathbf{s}_{j}$
  and thus $i^* = \hat{i}$ for Algorithm \ref{alg:simp} and $\tau = \tau_{\text{simplex}}$.  
\end{proof}

\section{Proof of Theorem \ref{thm:tau_linop}}
\label{app:tau_linop_pf}

\tauLinOp*

\begin{proof}
    Let $\widehat{\mathbf{x}} = |\mathbf{x}|/\alpha + \mu$ and $\widehat{\mathbf{s}} = \mathbf{s}/\alpha + \mu$, where $\mathbf{s}$ is a permutation of $|\mathbf{x}|$ such that $\mathbf{s}_1 \geq \mathbf{s}_2 \geq \cdots \geq \mathbf{s}_m \geq 0$ and $\mathbf{s}_{m+1} = 0.$  By Algorithm \ref{alg:base_min} there exists an $i$ such that $\mathbf{s}_{i+1} < \tau \leq \mathbf{s}_{i}$, where $\tau = (\sum_{k=1}^{i}\mathbf{s}_k - \alpha)/i$.  Since 
    \begin{equation*}
        \begin{split}
            \mathbf{s}_{i+1} < \tau \leq \mathbf{s}_{i} & \Leftrightarrow
            \mathbf{s}_{i+1}/\alpha + \mu < \tau/\alpha + \mu \leq \mathbf{s}_{i}/\alpha + \mu \\
            & \Leftrightarrow \widehat{\mathbf{s}}_{i+1} < \tau/\alpha + \mu \leq \widehat{\mathbf{s}}_{i} \\   
        \end{split} 
    \end{equation*}
    and $\tau/\alpha + \mu = (\sum_{k=1}^{i}\widehat{\mathbf{s}}_k - 1)/i$, it is clear that $|| \textbf{prox}_{\, ||\cdot||_\infty}(|\mathbf{x}|/\alpha + \mu) ||_\infty = \tau/\alpha + \mu$.  
\end{proof}

\section{Proof of Theorem \ref{thm:tau_cont}}
\label{app:tau_cont_pf}

\tauCont*

\begin{proof}  
    Let $\widehat{\mathbf{x}} \in \mathbb{R}^m$ be chosen arbitrarily and let  $\widehat{\mathbf{s}}$ be a permutation of $|\widehat{\mathbf{x}}|$ such that $\widehat{\mathbf{s}}_1 \geq \widehat{\mathbf{s}}_2 \geq \cdots \geq \widehat{\mathbf{s}}_m \geq 0$ and $\widehat{\mathbf{s}}_{m+1} = 0$, and $\widehat{\tau} = \tau(\widehat{\mathbf{s}})$. Fix $\epsilon>0$, and let $\mathbf{s}$ be the vector $\widehat{\mathbf{s}}$ with small perturbations in its elements, i.e., $\forall j \in [m]$, $|\mathbf{s}_j - \widehat{\mathbf{s}}_j| < \delta$. We shall find $\delta=\delta(\epsilon)>0$ so that $|\tau(\mathbf{s})-\tau(\widehat{\mathbf{s}})|<\epsilon$. We prove the result in two cases. \newline
    Case 1: $\widehat{\mathbf{s}}_{i+1} < \widehat{\tau} < \widehat{\mathbf{s}}_{i}$ for some $i \in [m]$. 
    Choose $\delta = \min(\epsilon, \frac{1}{2}(\widehat{\tau}-\widehat{\mathbf{s}}_{i+1}), \frac{1}{2}(\widehat{\mathbf{s}}_{i} - \widehat{\tau}))$.  Then $ \max(\mathbf{s}_{i+1}, \widehat{\mathbf{s}}_{i+1}) < \tau(\mathbf{s}) < \min(\mathbf{s}_i, \widehat{\mathbf{s}}_{i})$, and thus,
    \[
        | \tau(\mathbf{s}) - \tau(\widehat{\mathbf{s}}) | = 
        \frac{1}{i} \left|  \sum_{k=1}^{i}\mathbf{s}_k - \alpha - \sum_{k=1}^{i}\widehat{\mathbf{s}}_k + \alpha \right| \leq \frac{1}{i}\sum_{k=1}^{i} |\mathbf{s}_k - \widehat{\mathbf{s}}_k| < \delta \leq \epsilon.
    \]

    \noindent Case 2: $\widehat{\tau} = \widehat{\mathbf{s}}_{i}$ for some $i \in [m]$. 
    
    Subcase 2.1: $ \widehat{\mathbf{s}}_{i+r+1} < \widehat{\mathbf{s}}_{i+r}=\cdots = \widehat{\mathbf{s}}_{i+1}=\widehat{\mathbf{s}}_i=\widehat{\tau} < \widehat{\mathbf{s}}_{i-1}$.  Let $\delta_{\text{min}} = \min (\widehat{\mathbf{s}}_{i-1} - \widehat{\mathbf{s}}_{i},\widehat{\mathbf{s}}_{i+r} - \widehat{\mathbf{s}}_{i+r+1})$ and choose $\delta = \min(\epsilon, \frac{1}{5}\delta_{\text{min}})$. By Theorem \ref{thm:same_t}, we can consider the index $i^*$ in Algorithm \ref{alg:simp} associated to vector $\mathbf{s}$, the perturbation of $\widehat{\mathbf{s}}$. We claim that $i-1\leq i^*\leq i+r$. By Lemma \ref{lem:repeat_t}, 
\[ 
\widehat{\mathbf{s}}_{i+r} = \widehat{\mathbf{s}}_{i} = \widehat{\tau} = \frac{1}{i+r}\left( \sum_{k=1}^{i+r} \widehat{\mathbf{s}}_k - \alpha \right) = 
\frac{1}{i-1}\left( \sum_{k=1}^{i-1} \widehat{\mathbf{s}}_k - \alpha \right). 
\] 
Hence,
\begin{equation*}
    \begin{split}
    \mathbf{s}_{i-1}-\frac{1}{i-1}\left( \sum_{k=1}^{i-1} \mathbf{s}_k -\alpha \right) & = 
    \widehat{\mathbf{s}}_{i-1}-\frac{1}{i-1}\left( \sum_{k=1}^{i-1} \widehat{\mathbf{s}}_k -\alpha \right) + (\mathbf{s}_{i-1}- \widehat{\mathbf{s}}_{i-1}) - \frac{1}{i-1}\sum_{k=1}^{i-1} (\mathbf{s}_k - \widehat{\mathbf{s}}_k) \\    
    & \geq \widehat{\mathbf{s}}_{i-1} - \widehat{\mathbf{s}}_{i} - 2\delta \\ & \geq \delta_{\text{min}}-2\delta \\
    & > 0. 
    \end{split}
\end{equation*}
Thus, by the condition on line 6 of Algorithm \ref{alg:simp}, $i^*\geq i-1$. 

Similarly,
\begin{equation*}
    \begin{split}
        \mathbf{s}_{i+r+1}-\frac{1}{i+r+1}\left( \sum_{k=1}^{i+r+1} \mathbf{s}_k -\alpha \right) & = \\
    \widehat{\mathbf{s}}_{i+r+1} - \frac{1}{i+r+1}\left( \sum_{k=1}^{i+r+1} \widehat{\mathbf{s}}_k -\alpha \right) & + (\mathbf{s}_{i+r+1}- \widehat{\mathbf{s}}_{i+r+1}) - 
    \frac{1}{i+r+1}\sum_{k=1}^{i+r+1} (\mathbf{s}_k - \widehat{\mathbf{s}}_k) \\
    & \leq \frac{i+r}{i+r+1}\widehat{\mathbf{s}}_{i+r+1} -\frac{i+r}{i+r+1}\widehat{\mathbf{s}}_{i+r} +2\delta \\
    & \leq -\frac{i+r}{i+r+1}\delta_{\text{min}}+2\delta \\
    & \leq -\frac{1}{2}\delta_{\text{min}}+2\delta \\
    & < 0,
    \end{split}
\end{equation*}
and thus $i^*<i+r+1$.  Consequently,
\begin{equation}
\label{eqn:tau_cont3}
\tau(\mathbf{s}) =  \frac{1}{i-1+p}\left( \sum_{k=1}^{i-1+p} \mathbf{s}_k - \alpha \right)  
\end{equation}
for some $p\in \{0,1,2,\ldots, r+1\}$.  Hence, by Lemma \ref{lem:repeat_t},
\begin{equation}
\label{eqn:tau_cont4}
| \tau(\mathbf{s}) - \tau(\widehat{\mathbf{s}}) | \leq 
\frac{1}{i-1+p} \sum_{k=1}^{i-1+p} |\mathbf{s}_k - \widehat{\mathbf{s}}_k| < \delta \leq \epsilon.
\end{equation}

Subcase 2.2: $ \widehat{\mathbf{s}}_{2+r} < \widehat{\mathbf{s}}_{1+r}=\cdots = \widehat{\mathbf{s}}_{2}=\widehat{\mathbf{s}}_1=\widehat{\tau}$.   Let $\delta_{\text{min}} = \widehat{\mathbf{s}}_{1+r} - \widehat{\mathbf{s}}_{2+r}$ and choose $\delta = \min(\epsilon, \frac{1}{5}\delta_{\text{min}})$.  Again we consider the index $i^*$ in Algorithm \ref{alg:simp} associated to vector $\mathbf{s}$. We claim that $1 \leq i^*\leq 1+r$.  The logic of Subcase 2.1 applies for $i=1$ except for 1) establishing the lower bound on $i^*$, which is true by Theorem \ref{thm:psi}, and 2) Equations \ref{eqn:tau_cont3} and \ref{eqn:tau_cont4} apply for $p \in [r+1]$.  

Subcase 2.3: $ 0 = \widehat{\mathbf{s}}_{m+1} = \widehat{\mathbf{s}}_{m}=\cdots = \widehat{\mathbf{s}}_{i}=\widehat{\tau} < \widehat{\mathbf{s}}_{i-1}$.  Let $\delta_{\text{min}} = \widehat{\mathbf{s}}_{i-1} - \widehat{\mathbf{s}}_{i}$ and choose $\delta = \min(\epsilon, \frac{1}{3}\delta_{\text{min}})$.  We claim that $i-1 \leq i^* \leq m$ for the index $i^*$ in Algorithm \ref{alg:simp} associated to vector $\mathbf{s}$.  The logic of Subcase 2.1 applies for $r=m-i$ except for establishing the upper bound on $i^*$, which is trivially true.

Subcase 2.4:  $ 0 = \widehat{\mathbf{s}}_{m+1} = \widehat{\mathbf{s}}_{m}=\cdots = \widehat{\mathbf{s}}_{1}=\widehat{\tau}$, i.e. $\widehat{\mathbf{s}} = \mathbf{0}$.  Choose $\delta = \min(\epsilon, \frac{\alpha}{m})$.  Then $|| \mathbf{s} ||_1 \leq \alpha$, and by Theorem \ref{thm:psi}, $\tau(\mathbf{s}) = 0$.  Thus,
\[
| \tau(\mathbf{s}) - \tau(\widehat{\mathbf{s}}) | = 0 < \delta \leq \epsilon.
\]

\noindent Since $\widehat{\mathbf{x}}$ was chosen arbitrarily in each case, $\tau(|\mathbf{x}|)$ is continuous.     
\end{proof}

\section{Proof of Theorem \ref{thm:zero_pad}}
\label{app:zero_pad_pf}

\zeroPad*

\begin{proof} 
    Since $\mathbf{x}_{\text{zero-pad}} = (\mathbf{x}_1, \mathbf{x}_2, ..., \mathbf{x}_m, 0, ..., 0)$, the computations in each iteration of Algorithm \ref{alg:base_min} will be the same for $\mathbf{x}_{\text{zero-pad}}$ as they are for $\mathbf{x}$ until the algorithm terminates. Thus, $\tau_{\text{zero-pad}} = || \textbf{prox}_{\alpha ||\cdot||_\infty}(\mathbf{x}_{\text{zero-pad}}) ||_\infty = \tau$. 
\end{proof}

\section{Data Dependent Network Performance}
\label{app:dens}

In Section \ref{sec:nn} we demonstrated that performance of the network is highly dependent on the vector data distribution.  Specifically, uniformly distributed $\mathcal{U}(0,1)$ vector data performed better than Gaussian $\mathcal{N}(0,1)$ vector data.  For vectors of length 1,000-2,000, the $\tau$ testing error for $\mathcal{U}(0,1)$ data was two orders of magnitude lower than that for $\mathcal{N}(0,1)$ data ($10^{-6}$ vs. $10^{-4}$), and for vectors of length 1,000-100,000 the difference was four orders of magnitude ($10^{-7}$ vs. $10^{-3}$).  To explore why $\mathcal{U}(0,1)$ vector data performed better than $\mathcal{N}(0,1)$ vector data, we investigated network performance with uniformly distributed $\mathcal{U}(0,a)$ vector data for $a = 10$ and $a=20$, and compared with that of $\mathcal{N}(0,1)$ vector data. 

We performed four additional experiments, which we denote D1-D4.  These experiments were structured exactly as experiments 1-6, with the only difference being the datasets used.  Dataset information is given in Table \ref{tab:dens_data}.  

\begin{table}[h]
    \caption{Summary of vector data for experiments D1-D4}
    \label{tab:dens_data}
    \begin{tabular}{cccl}
        \toprule
        \textbf{Ex} & \textbf{$\mathcal{U}(0,10)$ vectors} & \textbf{$\mathcal{U}(0,20)$ vectors} & \textbf{Vector Lengths}  \\
        \midrule
        D1 & 10,000 & 0 & 1,000 - 2,000  \\
        D2 & 10,000 & 0 & 1,000 - 100,000  \\
        D3 & 0 & 10,000 & 1,000 - 2,000  \\
        D4 & 0 & 10,000 & 1,000 - 100,000  \\
        \botrule
    \end{tabular}
\end{table}
 
Learning curves for experiments D1-D4 are provided in Figure \ref{fig:density_lc}, including four plots that were previously provided in Section \ref{sec:nn}: 1) the learning curves for $\mathcal{U}(0,1)$ vector data, and 2) the learning curves for $\mathcal{N}(0,1)$ vector data.  In Figure \ref{fig:density_lc} we see that for uniformly distributed vector data, performance decreases as variance increases.  For vectors of length 1,000-2,000, performance is similar for $\mathcal{U}(0,20)$ and $\mathcal{N}(0,1)$ data.   However, for vectors of length 1,000-100,000, the $\tau$ testing error for $\mathcal{U}(0,20)$ data is lower than that for $\mathcal{N}(0,1)$ data by two orders of magnitude ($10^{-5}$ vs. $10^{-3})$.   

\begin{sidewaysfigure}
    \centering
    \includegraphics[scale=0.5]{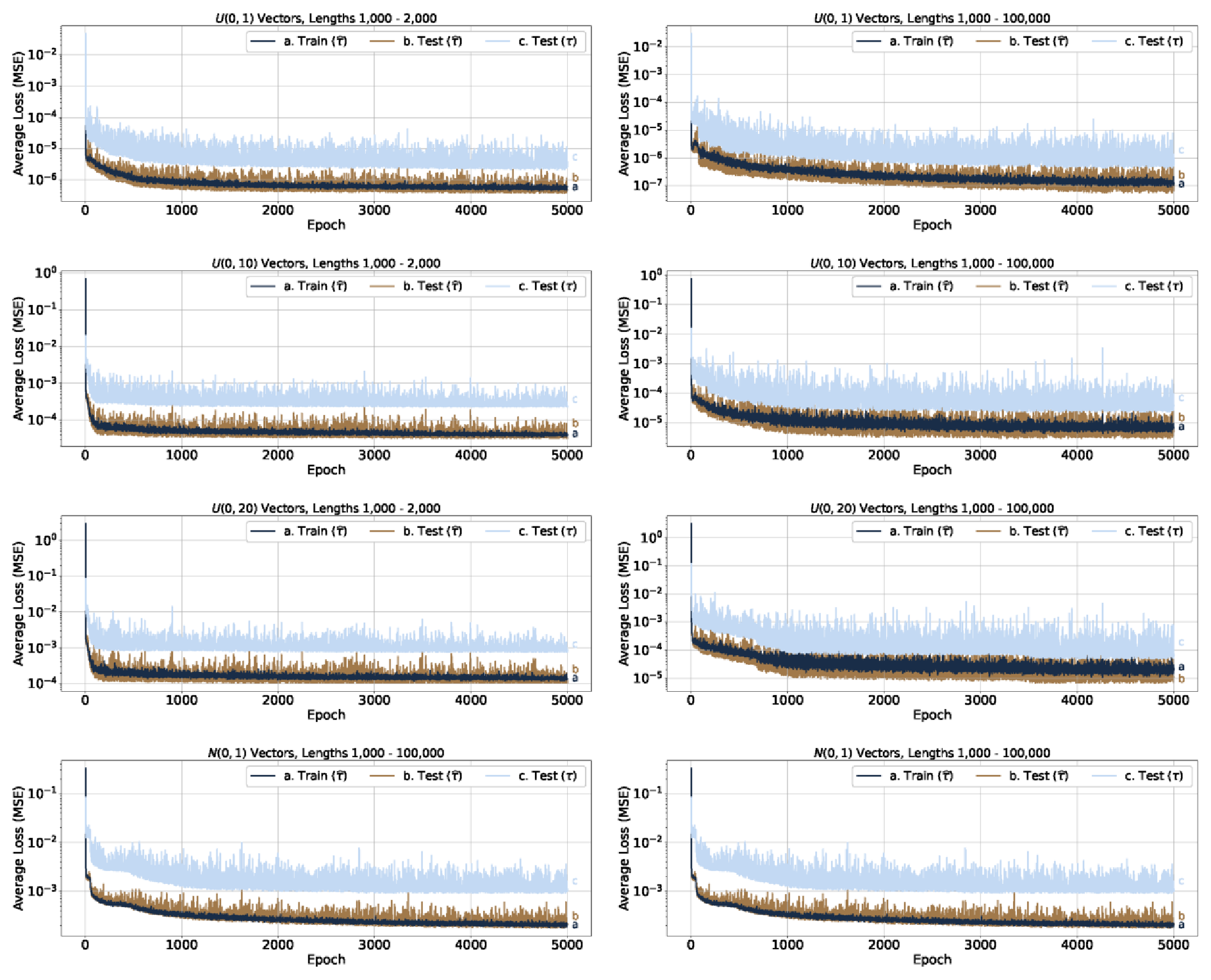}
    \caption{Neural network learning curves for $\mathcal{U}(0,1)$ vector data (experiments 3-4),  $\mathcal{U}(0,10)$ vector data (experiments D1-D2),  $\mathcal{U}(0,20)$ vector data (experiments D3-D4), and $\mathcal{N}(0,1)$ vector data (experiments 1-2)}
    \label{fig:density_lc}
\end{sidewaysfigure}

In Figure \ref{fig:density_fi} we provide the feature importances as given by saliency for the best models from experiments D1-D4, along with those from Section \ref{sec:nn} for experiments on $\mathcal{U}(0,1)$ and $\mathcal{N}(0,1)$ vector data for comparison.  For all uniformly distributed vector data experiments except that on $\mathcal{U}(0,1)$ vectors of lengths 1,000-2,000, the maximum is by far the most important feature.  For $\mathcal{N}(0,1)$ vector data, the maximum is relatively important for vectors of lengths 1,000-2,000 and relatively unimportant for vectors of lengths 1,000-100,000.  Since the computation of $\tau$ depends on the largest $i$ elements of $|\mathbf{x}|$ (Algorithm \ref{alg:base_min}), we hypothesize based on these numerical results that the performance of the network for a given $\mathbf{x}$ could be related to the density of elements in $|\mathbf{x}|$ near the maximum of $|\mathbf{x}|$.  However, more investigation is needed and is a potential area for future work.

\begin{figure}[H]
    \centering
    \includegraphics[width=1.0\textwidth]{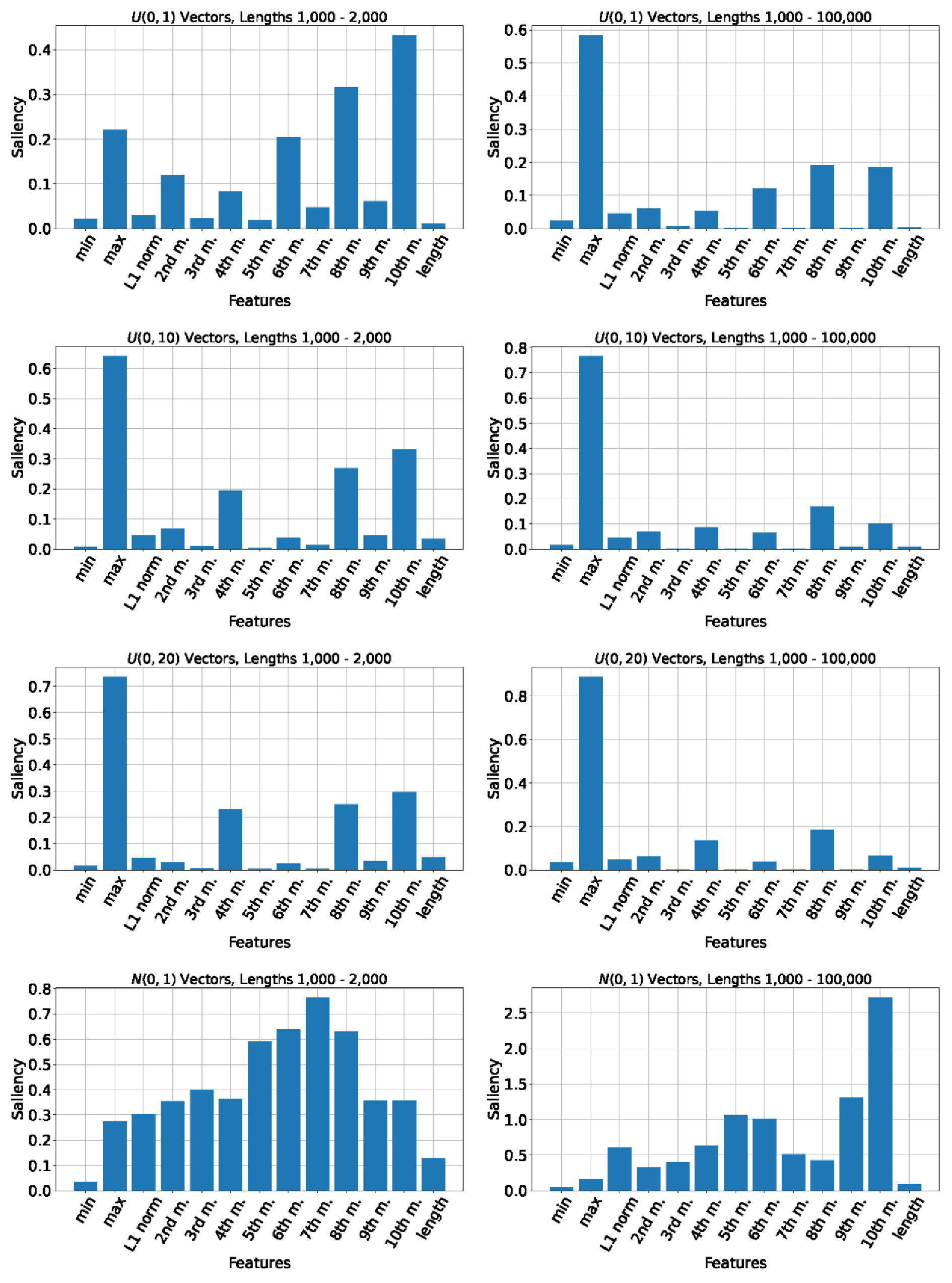}
    \caption{Feature importances for the best models from experiments 3-4, D1-D4, and 1-2}
    \label{fig:density_fi}
\end{figure}

\end{appendices}

\bibliography{refs}

\end{document}